\documentclass[10pt,twocolumn,twoside]{IEEEtran}
\usepackage{amsmath,amssymb,euscript ,yfonts,psfrag,latexsym,dsfont,graphicx,bbm,color,amstext,wasysym,subfig,parskip,pdfsync}
\usepackage{epstopdf}
\usepackage{flushend}
\usepackage{amsthm}
\usepackage{mathtools}
\usepackage{makecell}
\graphicspath{{./},{./figures/},{../figures/}}

\newcounter{plm}
\newtheorem{theorem}[plm]{Theorem}
\newtheorem{lemma}[plm]{Lemma}

\newtheorem{remark}[plm]{Remark}
\newtheorem{proposition}[plm]{Proposition}

\begin{document}
\def\spacingset#1{\def\baselinestretch{#1}\small\normalsize}

\title{The Covariance Extension Equation:\\ A Riccati-type Approach to Analytic Interpolation}

\author{Yufang Cui,~\IEEEmembership{Student Member,~IEEE} and Anders Lindquist,~\IEEEmembership{Life Fellow,~IEEE}
\thanks{Y.\ Cui is with the Department of Automation, Shanghai Jiao Tong University, Shanghai, China; {email: cui-yufang@sjtu.edu.cn} and A.\ Lindquist is with the Department of  Automation and the School of Mathematical Sciences, Shanghai Jiao Tong University, Shanghai, China; {email: alq@kth.se}}
} 

\maketitle

\setlength{\parindent}{15pt}
\parskip 6pt
\def\spacingset#1{\def\baselinestretch{#1}\small\normalsize}
\newcommand{\E}{\operatorname{E}}
\newcommand{\bX}{\mathbf X}
\newcommand{\bH}{\mathbf H}
\newcommand{\bA}{\mathbf A}
\newcommand{\bB}{\mathbf B}
\newcommand{\bN}{\mathbf N}

\newcommand{\bHo}{{\stackrel{\circ}{\mathbf H}}}

\newcommand{\bHom}{{\stackrel{\circ\,\,}{{\mathbf H}_{t}}}{\hspace*{-13pt}\phantom{H}}^-}
\newcommand{\bHop}{{\stackrel{\circ\,\,}{{\mathbf H}_{t}}}{\hspace*{-13pt}\phantom{H}}^+}
\newcommand{\bHotwo}{{\stackrel{\circ\,\,\,\,\,}{{\mathbf H}_{t_2}}}{\hspace*{-16pt}\phantom{H}}^-}
\newcommand{\bHomn}{{\stackrel{\circ}{{\mathbf H}}}{\hspace*{-9pt}\phantom{H}}^-}
\newcommand{\bHopn}{{\stackrel{\circ}{{\mathbf H}}}{\hspace*{-9pt}\phantom{H}}^+}

\newcommand{\EbHom}{\E^{{\stackrel{\circ\,\,}{{\mathbf H}_{t}}}{\hspace*{-11pt}\phantom{H}}^-}}
\newcommand{\EbHop}{\E^{{\stackrel{\circ\,\,}{{\mathbf H}_{t}}}{\hspace*{-10pt}\phantom{H}}^+}}

\spacingset{.97}

\begin{abstract} 
Analytic interpolation problems with rationality and derivative constraints are ubiquitous in systems and control. This paper provides a new method for such problems, both in the scalar and matrix case, based on a non-standard Riccati-type equation. The rank of the solution matrix is the same as the degree of the interpolant, thus providing a natural approach to model reduction.  A  homotopy continuation method is presented and applied to some problems in modeling and robust control. We also address a question  on the positive degree of a covariance sequence originally posed by Kalman.
\end{abstract}


\newcommand{\mR}{{\mathbb R}}
\newcommand{\mZ}{{\mathbb Z}}
\newcommand{\mN}{{\mathbb N}}
\newcommand{\mE}{{\mathbb E}}
\newcommand{\mC}{{\mathbb C}}
\newcommand{\mD}{{\mathbb D}}
\newcommand{\bU}{{\mathbf U}}
\newcommand{\bW}{{\mathbf W}}
\newcommand{\cF}{{\mathcal F}}

\newcommand{\trace}{{\rm trace}}
\newcommand{\rank}{{\rm rank}}
\newcommand{\Real}{{\Re}e\,}
\newcommand{\half}{{\frac12}}

\spacingset{1}

\section{Introduction}\label{sec:intro}

Analytic interpolation problems abound in systems and control, occurring in spectral estimation, robust control, system identification and signal processing, to mention a few. In the scalar case, the most general problem formulation goes as follows.
Given $m+1$ distinct complex numbers $z_0,z_1,\dots,z_m$ in the open unit disc $\mathbb{D}:=\{z \mid |z|<1\}$, consider the problem  to find a real Carath\'eodory function mapping the unit disc $\mathbb{D}$ to the open right half-plane, i.e., a real function $f$ that is analytic in $\mathbb{D}$ and satisfies $\text{Re}\{f(z)\}>0$ there, and which in addition satisfies the interpolation conditions 
\begin{align}
\label{interpolation}
 \frac{f^{(k)}(z_{j})}{k!}=w_{jk},\quad&j=0,1,\cdots,m,   \\
    &   k=0,\cdots n_{j}-1 \notag
\end{align}
where $f^{(k)}$ is the $k$:th derivative of $f$, and the interpolation values $\{w_{jk}; j=0,1,\cdots,m, k=0,\cdots n_{j}-1\}$ are complex numbers in the open right half plane $\mathbb{C}^+$ that occur in conjugate pairs.  In addition we impose the complexity constraint that the interpolant  $f$ is rational of degree at most 
\begin{equation}
\label{deg(f)}
n:=\sum_{j=0}^{m}n_j -1.
\end{equation}
To simplify calculations, we  normalize the problem by setting $z_0=0$ and $ f(0)=\tfrac{1}{2} $, which can be achieved through a simple M{\"o}bius transformation. Since $f$ is a real function,  $f^{(k)}(\bar{z}_j)/ k!=\bar{w}_{jk}$ is an interpolation condition whenever $f^{(k}(z_j)/ k!=w_{jk}$ is.

For $m=0$ and $n_0=n+1$,  this becomes the {\em rational covariance extension problem\/} introduced by Kalman \cite{Kalman-81} and completely solved in steps in \cite{Gthesis,G87,BLGuM,BLpartial,BGuL}. This problem, which is equivalent to determining a rational positive real function of prescribed maximal degree given a partial covariance sequence, is a basic problem in signal processing and speech processing \cite{b12} and system identification \cite{b13,LPbook}.

With $n_0=n_1=\dots=n_m=1$,  we have the regular {\em Nevanlinna-Pick interpolation problem with degree constraint\/} \cite{b15,b1,BLkimura} occurring in robust control \cite{DFT}, high-resolution spectral estimation \cite{b2,GL1}, simultaneous stabilization \cite{Ghosh} and many other problems in systems and control. In fact, the Nevanlinna-Pick interpolation problem to find a Carath\'eodory function that interpolates the given data was early used in systems and control \cite{b9,b10}. The general Nevanlinna-Pick interpolation problem described above, allowing derivative constraints, was studied  in \cite{blomqvist}, motivated by $H^\infty$ control problems with multiple unstable poles and/or zeros in the plant. 
Such problems could not be handled by a classical interpolation approach \cite[p. 18]{GreenLimebeer}. 

The early work on the rational covariance extension problem \cite{Gthesis,G87,BLGuM} had nonconstructive proofs based on topological degree theory. A first attempt  to provide an algorithm was presented in \cite{BLpartial}, where a new nonstandard Riccati-type equation called the Covariance Extension Equation (CEE) was introduced. However, this approach was completely superseded by a convex optimization approach \cite{BGuL,b1}, and thus abandoned. However, in a brief paper \cite{b6}, it was shown that the regular Nevanlinna-Pick interpolation problem with degree constraint could also be solved by the Covariance Extension Equation, and thus it was shown that CEE  is universal in the sense that it can be used to solve more general analytic interpolation problems by only changing certain parameters. This idea was then used in \cite{CLccdc} to attach the general problem presented above. A first attempt to generalize this method to multivariable analytic interpolation problems was then made in \cite{CLcdc19}, and we shall pursue this inquiry in this paper. 

To provide basic insight into ideas behind the CEE approach, in Section~\ref{sec:prel} we shall review its application to the rational covariance extension problem, and also bring up the issue of the importance to distinguish between positive and algebraic degree of partial covariance sequences. This is important since  the CEE approach provides a simple tool for model reduction. In Section~\ref{sec:generalscalar} we deal with the general scalar problem formulated above and provide a numerical algorithm based on  homotopy continuation in the style of \cite{BFL}. Section~\ref{sec:multivariable} is devoted to the multivariable generalization, which turns out to be a challenging problem. The results fall somewhat short of what the scalar case promises, and, given some results in \cite{Takyar}, we suspect that this is due to problems introduced by the nontrivial Jordan structure of the multivariable case. In Section~\ref{sec:examples} we illustrate our theory with some numerical examples, and finally in Section~\ref{sec:conclusions} we provide some conclusions. 

\section{Preliminaries on Covariance Extension}\label{sec:prel}

To clarify basic concepts and set notation we first develop and review basic theory for the the special case 
that $m=0$, $z_0=0$ and 
\begin{equation}
\label{wc}
w_{0k}=c_k, \quad k=0,1,\dots,n,
\end{equation}
where the Toeplitz matrix
\begin{equation}
\label{Toepliz}
T=\begin{bmatrix}c_0&c_1&\cdots&c_n\\c_1&c_0&\cdots&c_{n-1}\\ \vdots&\vdots&\ddots&\vdots\\c_n&c_{n-1}&\cdots&c_0
\end{bmatrix}
\end{equation}
is positive definite. A sequence $(c_0,c_1,\dots,c_n)$ with the property $T>0$ is called a {\em positive covariance sequence}. To normalize the problem we set $c_0=\tfrac12$.

\subsection{The rational covariance extension problem}

 If $f$ is a Carath\'eodory function, then
\begin{equation}
\label{phi+}
\phi_+(z):= f(z^{-1})
\end{equation}
is a {\em positive real\/} function. The problem is then reduced  to finding a rational positive real function 
\begin{equation}
\label{covexpansion}
\phi_+(z)=\tfrac{1}{2}+c_{1}z^{-1}+c_{2}z^{-2}+c_{3}z^{-3}+\cdots ,
\end{equation}
of degree at most $n$ for which only the first $n$ coefficients $c_1, c_2,\dots,c_n$ are specified. This is the rational covariance extension problem. In fact,
\begin{equation}\label{phi+2phi}
\phi(z):=\phi_+(z)+\phi_+(z^{-1})=\sum_{k=-\infty}^{\infty}c_{k}z^{-k}>0, \quad z\in \mathbb{T},
\end{equation}
where $\mathbb{T}$ is the unit circle $\{ z=e^{i\theta}\mid 0\leq\theta <2\pi\}$. Hence $\phi$ is a power spectral density, and therefore there is a minimum-phase spectral factor $v(z)$ such that 
\begin{equation}
\label{spectralfactor}
v(z)v(z^{-1})=\phi(z).
\end{equation}
It is well-known \cite{LPbook} that passing normalized white noise $\{ u(t)\}_{t\in\Bbb{Z}}$ through a shaping filter with transfer function $v(z)$, i.e., 
\[
\text{white noise}\stackrel{u}\negmedspace{\longrightarrow}\fbox{$v(z)$}\negmedspace\stackrel{y}
{\longrightarrow} 
\]
until steady state, the output  $\{ y(t)\}_{t\in\Bbb{Z}}$ is a stationary process with power spectral density $\phi(e^{i\theta})$, $\theta\in [-\pi,\pi]$. Moreover, the coefficient $(c_0,c_1,c_2,\dots)$ are the covariance lags
\begin{equation}
\label{ }
c_k=\E\{y(t+k)y(t)\}.
\end{equation}

Since $\phi_+(z)$ is rational of degree at most $n$, it can be represented as 
\begin{subequations}
\label{abphi+}
\begin{equation}
\label{ab2phi+}
\phi_+(z)=\frac12\frac{b(z)}{a(z)},
\end{equation}
where 
\begin{align}
    a(z)&=z^n+a_1z^{n-1}+\dots +a_n \label{a}\\
    b(z)&=z^n+b_1z^{n-1}+\dots +b_n \label{b}
\end{align}
\end{subequations}
are {\em Schur polynomials}, i.e., monic polynomials with all its roots in the open unit disc. 
Consequently, a simple calculation shows that
\begin{equation}
\label{v}
v(z)=\rho\frac{\sigma(z)}{a(z)},
\end{equation}
where $\rho >0$ and
\begin{equation}
\label{sigma}
\sigma(z)=z^n+\sigma_1z^{n-1}+\dots +\sigma_n 
\end{equation} 
is a Schur polynomial satisfying
\begin{equation}
\label{absigma}
a(z)b(z^{-1})+b(z)a(z^{-1})=2\rho^2\sigma(z)\sigma(z^{-1}).
\end{equation}

\begin{theorem}\label{thm:covext}
Let $(c_0,c_1,\dots,c_n)$ be a positive covariance sequence. Then, given any Schur polynomial \eqref{sigma}, there is one and only one Schur polynomial \eqref{a} and $\rho>0$ such that \eqref{v}
 is a shaping filter for $(c_0,c_1,\dots,c_n)$. The mapping from $\sigma$ to $(a,\rho)$ is diffeomorphism. 
\end{theorem}

The existence part of Theorem~\ref{thm:covext} was proved in \cite{Gthesis} (also see \cite{G87}) and the rest of the theorem in \cite{BLGuM}. Note that $\sigma(z)$ and $a(z)$ may have common roots, so the degree of $v(z)$ might be less than $n$. Via \eqref{absigma} there is a one-one correspondence between $v$ and $\phi_+$, and they have the same degree.

For each parameter $\sigma$ there is a convex optimization problem solving for $(a,\rho)$, which first appeared in \cite{BGuL} (also see \cite{SIGEST,BEL}), but here we shall consider a different method of solution described next. 

\subsection{Covariance Extension  Equation}\label{sec:CEE}

Given the parameter polynomial \eqref{sigma}, we introduce
\begin{equation}\label{sigmaGammah}
\sigma = \bmatrix \sigma_1\\ \sigma_2\\ \vdots\\ \sigma_n \endbmatrix,
\;
\Gamma  = \begin{bmatrix} -\sigma_1  & 1 & 0 & \cdots & 0\\ -\sigma_2 
& 0 & 1 & \cdots & 0\\
\vdots& \vdots & \vdots & \ddots & \vdots\\ -\sigma_{n-1} & 0 & 0 &
\cdots & 1\\ -\sigma_n & 0 & 0 & \cdots & 0 
\end{bmatrix} ,
\;h=
\begin{bmatrix}1\\0\\ \vdots\\0 \end{bmatrix}. 
\end{equation}
Moreover, we represent the covariance data in terms of the first $n$ coefficients in the expansion 
\begin{subequations}\label{uUconstr}
\begin{equation}\label{expansion}
\begin{split}
&\frac{z^n}{z^n + c_1 z^{n-1}+\dots + c_n}\\ 
&\qquad\qquad\quad= 1 - u_1 z^{-1} - u_2z^{-2} - u_3 z^{-3} - \dots
\end{split}
\end{equation}
about infinity and define 
\begin{equation}\label{Uu}
u = \begin{bmatrix} u_1\\u_2\\ \vdots  \\u_n \end{bmatrix} ,
\qquad 
U  =  \begin{bmatrix} 0&  &   &  & \\ u_1&0 & & &\\ u_2&u_1& & & \\
\vdots &\vdots& \ddots&  &\\
u_{n-1}&u_{n-2}&\cdots&u_1&0\end{bmatrix}
\end{equation} 
\end{subequations} 
and the function $g: \Bbb{R}^{n\times n}\to \Bbb{R}^n$ given by
\begin{equation}\label{g(P)} 
g(P)= u +U\sigma + U\Gamma Ph .
\end{equation}
The {\em Covariance Extension Equation\/} (CEE), introduced in \cite{BLpartial}, is the nonstandard
 Riccati equation  
\begin{equation} \label{CEE}
 P = \Gamma (P-Phh'P) \Gamma' + g(P)g(P)' ,
\end{equation} 
where $^\prime$ denotes transposition.  The following theorem was proved in \cite{BLpartial}.


\begin{theorem}\label{thm:CEE}
Let $(c_0,c_1,\dots,c_n)$ be a positive  covariance sequence. Then, for each Schur polynomial \eqref{sigma}, there is a unique symmetric, positive semi-definite  solution $P$ of CEE satisfying $h'Ph<1$. Moreover, for each $\sigma$ there is a unique shaping filter \eqref{v}  for $(c_0,c_1,\dots,c_n)$ and a corresponding positive real function \eqref{ab2phi+},
where $a(z)$, $b(z)$ and $\rho$ are given in terms of the corresponding $P$ by
\begin{equation}
\label{abrho}
\begin{split}
&a=(I-U)(\Gamma Ph+\sigma)-u,\\
&b=(I+U)(\Gamma Ph+\sigma)+u,\\
&\rho=\sqrt{1-h'Ph} .
\end{split}
\end{equation} 
\end{theorem}
Here $a:=(a_1,a_2,\dots,a_n)'$ and $b:=(b_1,b_2,\dots,b_n)'$. Finally
\begin{equation}
\label{deg=rank}
\deg v =\deg\phi_+ =\rank P .
\end{equation}
Note that $P$ looses rank, i.e., has rank less than $n$, only on a thin (lower-dimensional) subset of parameters $u$ \cite{BLpartial}.

\subsection{Algebraic and positive degree}

For the moment, let $\phi_+(z)$ be any rational function of degree $d$, not necessarily positive real, given by \eqref{covexpansion}. Then it has a representation \eqref{abphi+} with $n$ replaced by $d$. Identifying coefficients of powers of $z$ in $b(z)=2\phi_+(z)a(z)$ as done in \cite{BLpartial}, we obtain
\begin{subequations}\label{c2ab}
\begin{equation}\label{c2b}
\begin{bmatrix}
b_1\\b_2\\ \vdots\\b_d
\end{bmatrix}
=2\begin{bmatrix}
c_1\\c_2\\ \vdots\\c_d
\end{bmatrix}
+
\begin{bmatrix}
1&&&\\2c_1&1&&\\ \vdots&\vdots&&\\
2c_{d-1}&2c_{d-2}&\dots&1
\end{bmatrix}
\begin{bmatrix}
a_1\\a_2\\ \vdots\\a_d
\end{bmatrix}
\end{equation}
for  nonnegative powers and 
\begin{equation}\label{Hankel} 
\begin{bmatrix} c_1& c_2& \cdots &c_{d}\\ 
c_2& c_3 &\cdots & c_{d+1}\\
\vdots &\vdots& \ddots& \vdots\\
 c_{d}&c_{d+1}&\cdots&c_{2d-1}\end{bmatrix}
\begin{bmatrix}
a_1\\a_2\\ \vdots\\a_d
\end{bmatrix}=
-\begin{bmatrix}
c_{d+1}\\c_{n+2}\\ \vdots\\c_{2d}
\end{bmatrix}
\end{equation}
\end{subequations}
for negative powers. The coefficient matrix in \eqref{Hankel} is a Hankel matrix that we denote $H_d$. By Kronecker's theorem \cite{Kalmanbook}, 
\begin{equation}
\label{Kronecker}
d:= \deg \phi_+(z) =\text{rank}\, H_\infty = \text{rank}\, H_d.
\end{equation}
Therefore $\phi_+(z)$ can be determined from a finite sequence $(c_0,c_1,\dots,c_n)$ of covariance lags, where  $n:=2d$.  We say that $d$ is the {\em algebraic degree\/} of $(c_0,c_1,\dots,c_n)$.

Therefore, at first blush, given a partial covariance sequence $(c_0,c_1,\dots,c_n)$,  we might assume that \eqref{c2ab} solves the rational covariance problem in a minimal-degree form. This idea underlies (at least the early work on) {\em subspace identification} \cite{Aoki,OverscheeDeMoor93,OverscheeDeMoor96}, where in general the biased ergodic estimates 
\begin{equation}
\label{ckest}
c_k=\frac{1}{N-k+1}\sum_{t=0}^{N-k}y_{t+k}y_t 
\end{equation}
would be used to insure that the corresponding Toeplitz matrix is positive definite. Then since 
\begin{displaymath}
c_k=h'F^{k-1}g
\end{displaymath}
where $\phi_+(z)$ has the realization 
\begin{equation}
\label{phi+real}
\phi_+(z)= \tfrac12 + h'(zI-F)^{-1}g,
\end{equation}
$(F,g,h)$ could be determined by minimal factorization of the Hankel matrix $H_d$. However, as pointed out in \cite{LP96}, this is incorrect and may lead to an $\phi_+(z)$ that is not positive real; also see \cite[Chapter 13]{LPbook}. In \cite{DLM} simple examples were given where subspace algorithms will fail.

In general we cannot achieve a solution $\phi_+$ to the rational covariance extension problem of a degree $d$ only half of $n$. By Theorem~\ref{thm:CEE}, the best we could do is
\begin{equation}
\label{p}
p:=\min_{\sigma}\rank\, P(\sigma),
\end{equation}
which we call the {\em positive degree\/} of the covariance sequence $(c_0,c_1,\dots,c_n)$. Since the algebraic degree can be determined from the rank of the Hankel matrix $H_d$ (also see \cite{Kalmanbook,GraggL}), in 1972 Kalman \cite{Kalmanprivate} posed the question whether there is a similar matrix-rank criterion for determining the positive degree. However, since then it has been shown \cite{BLpartial}  that for any $p$ between $[\tfrac{n}{2}]$ and $n$ there is an open set in $\mathbb{R}^n$ of covariance sequences $(c_1,c_2,\dots,c_n)$ for which $p$ is the positive degree. Hence it seems that we cannot get a better criterion than \eqref{p}.

\section{The general scalar problem}\label{sec:generalscalar}

Next we show that the Covariance Extension Equation is universal in the sense that it also solves the general analytic interpolation problem stated in the introduction, by merely adopting the parameters $(u,U)$ to the new interpolation data. 

\subsection{Some stochastic realization theory}\label{subsec:matrixrealization} 

We express the realization \eqref{phi+real} of $\phi_+(z)$ in the observable canonical form, where $h$ is defined as in \eqref{sigmaGammah},
\begin{equation}
\label{F}
F= \begin{bmatrix} -a_1  & 1 & 0 & \cdots & 0\\
-a_2 & 0 & 1 & \cdots & 0\\
\vdots& \vdots & \vdots & \ddots & \vdots\\
-a_{n-1} & 0 & 0 &\cdots & 1\\ 
-a_n & 0 & 0 & \cdots & 0 
\end{bmatrix}= J-ah' ,
\end{equation}
$J$ is the upward shift matrix, and $g$ is an $n$-vector to be determined. Note that this need not be a minimal realization, as there could be cancellations of common zeros of $a(z)$ and $b(z)$. 

\begin{lemma}\label{lem:ab2g}
The vector $g$ in \eqref{phi+real} is given by 
\begin{equation}
\label{ab2g}
g=\frac12 (b-a).
\end{equation}
\end{lemma}

\begin{proof}
From \eqref{ab2phi+} and \eqref{phi+real} we have
\begin{displaymath}
\frac{b(z)}{a(z)}=1 +2h'(zI-F)^{-1}g,
\end{displaymath}
to which we apply the matrix inversion lemma to obtain
\begin{displaymath}
\begin{split}
\frac{a(z)}{b(z)}&=1 -2h'(2gh'+zI-F)^{-1}g\\
&=1 -2h'[zI-(J-ah' -2gh')]^{-1}g.
\end{split}
\end{displaymath}
Consequently, since $b(z)$ is the denominator polynomial, we must have $a+2g=b$, from which \eqref{ab2g} follows.
\end{proof}

As a preliminary, let us review some facts from stochastic realization theory \cite[Section 6]{LPbook}. In view of \eqref{phi+2phi} and \eqref{phi+real}, the spectral density $\phi(z)$ may be written
\begin{subequations}
\begin{equation}
\label{M2phi}
\phi(z) = \begin{bmatrix} h'(zI-F)^{-1}&1\end{bmatrix}M(P)\begin{bmatrix}(z^{-1}I-F')^{-1}h\\1\end{bmatrix}
\end{equation}
for any symmetric $n\times n$ matrix $P$, where
\begin{equation}
\label{M(P)}
M(P)=\begin{bmatrix} P-FPF'&g-FPh\\g'-h'PF'&1-h'Ph   \end{bmatrix}.   
\end{equation}
\end{subequations}
As a straight-forward calculation shows, the left member of \eqref{M2phi} does not depend on $P$, as all terms containing $P$ cancel. However, $M(P)$ does depend on $P$, and, by the Positive Real Lemma (see, e.g. \cite[p. 200]{LPbook}), $\phi_+$ is positive real if and only if there is a $P$ such that 
\begin{equation}
\label{M(P)geq0}
M(P)\geq 0.
\end{equation}
In this case, $P$ must be positive semidefinite, and there is a minimum-rank factorization
\begin{equation}
\label{M(P)factors}
M(P)=\begin{bmatrix} k\\\rho\end{bmatrix}\begin{bmatrix} k'&\rho\end{bmatrix},
\end{equation}
where $k\in\mathbb{R}^n$ and $\rho\in\mathbb{R}$. Together with \eqref{M2phi} this yields \eqref{spectralfactor} with the spectral factor
\begin{equation}
\label{wreal}
v(z)=\rho + h'(zI-F)^{-1}k .
\end{equation}
There is a unique minimal symmetric solution of \eqref{M(P)geq0} in the ordering $\geq$ of symmmetric matrices, and from now on $P$ will denote precisely this solution.  Then \eqref{wreal} is the minimum-phase spectral factor with all poles and zeros in the open unit disc, i.e., \eqref{wreal}  is precisely \eqref{v}.  Moreover, from \eqref{M(P)factors} we also have
\begin{subequations}
\begin{align} 
\label{}
   P &=FPF'+kk'  \label{Lyapunov}  \\
  g  &= FPh+\rho k \label{Pk2g}\\
  \rho^2&= 1-h'Ph  , \label{rho}
\end{align}
\end{subequations}
from which we have the algebraic Riccati equation
\begin{equation}
\label{Riccati}
P=FPF'+(g-FPh)(1-h'Ph)^{-1}(g-FPh)' .
\end{equation}
Note that $\rho$ must be nonzero, or otherwise $v(z)$ would be identically zero by \eqref{v}, and thus the same would hold for the spectral density $\phi(z)$.
Therefore, in view of \eqref{rho},  
\begin{equation}
\label{h'Ph<1}
h'Ph<1.
\end{equation}

Since all eigenvalues of $F$ lie in the open unit disc, the Lyapunov equation \eqref{Lyapunov} has a unique solution
\begin{displaymath}
P=\sum_{j=0}^\infty F^jkk'(F')^j \geq 0
\end{displaymath}
 \cite[Proposition B.1.19, B.1.20]{LPbook}.
If $(F,k)$ is a reachable pair so that \eqref{wreal} is a minimal realization, then $P>0$ \cite[Proposition B.1.20]{LPbook}. If $\rank\, P=r<n$, there is a transformation $T$ and a positive definite $r\times r$ matrix $P_1$ such that 
\begin{displaymath}
TPT'=\begin{bmatrix}P_1&0\\0&0\end{bmatrix}.
\end{displaymath}
Setting
\begin{displaymath}
TFT^{-1} =\begin{bmatrix}F_{11}&F_{12}\\F_{21}&F_{22}\end{bmatrix},
\quad Tk=\begin{bmatrix}k_1\\k_2\end{bmatrix}, \quad   (T')^{-1}h=\begin{bmatrix}h_1\\h_2\end{bmatrix} ,
\end{displaymath}
it follows from \eqref{Lyapunov} that $F_{21}=0$ and $k_2=0$, and hence a straightforward calculation yields the minimal realization
\begin{displaymath}
v(z)=h_1'(zI-F_{11})^{-1}k_1 +\rho
\end{displaymath}
of degree $r=\rank\, P$. Moreover,
\begin{displaymath}
a(z)=\det (zI-F)^{-1}=\det (zI-F_{11})^{-1}\det (zI-F_{22})^{-1},
\end{displaymath}
so $\det (zI-F_{22})^{-1}$ must be the common factor in $\sigma(z)$ and $a(z)$ that is canceled. In view of \eqref{absigma}, $b(z)$ has the same common factor which is canceled in \eqref{ab2phi+}, and hence degree of $\phi_+(z)$  is also $r$. Thus 
\begin{equation}
\label{degree=rank}
\deg \phi_+(z)=\deg f=\rank\, P.
\end{equation} 
It is important to note that $P$ looses rank on a  thin set where zero cancelations occur. However, by considering the singular values of $P$, we can determine whether $P$ is close to being singular, which can then be  used for approximate model reduction.

\begin{remark}
Note that the  algebraic Riccati equation  \eqref{Riccati} is different from that of Kalman filtering. Indeed, if $\hat{x}(t)$ is the steady-state Kalman filter estimate of a stationary state process $x(t)$, then the  algebraic Riccati equation of Kalman filtering solves for the error covariance matrix
\begin{displaymath}
\Sigma:=\E\{[x(t)-\hat{x}(t)][x(t)-\hat{x}(t)]'\}=\Pi -P.
\end{displaymath}
where $\Pi:= \E\{x(t)x(t)'\}$, and  $P:=\E\{\hat{x}(t)\hat{x}(t)'\}$ is the matrix $P$ in our present setting {\rm \cite[Section 6.9]{LPbook}}.
\end{remark}

\begin{lemma}\label{lem:P2g}
The vectors $g$ and $k$ in \eqref{phi+real} and \eqref{wreal} are given by
\begin{subequations}\label{gk}
\begin{align}
   g &=\Gamma Ph+\sigma -a   \label{P2g}\\
   k & =\rho(\sigma -a), \label{P2k}
\end{align}
\end{subequations}
where $P$ is the minimal symmetric solution of \eqref{M(P)geq0}, or equivalently \eqref{Riccati}.
\end{lemma}

 \begin{proof}
In the same way as in the proof of Lemma~\ref{lem:ab2g}, the matrix inversion lemma yields
\begin{displaymath}
\begin{split}
\frac{a(z)}{\sigma(z)}&= 1-h'(\rho^{-1}kh' +zI-F)^{-1}\rho^{-1}k\\
&= 1-h'[zI-J+(a-\rho^{-1}k)h']^{-1}\rho^{-1}k.
\end{split}
\end{displaymath}
and consequently
\begin{displaymath}
J -(a-\rho^{-1}k)h'=J-\sigma h'=\Gamma,
\end{displaymath}
since the denominator is $\sigma(z)$. Hence \eqref{P2k} follows. Moreover, $\Gamma =F-\rho^{-1}k$, which, together with \eqref{P2k}, yields
\begin{displaymath}
(1-h'Ph)(\sigma -a)=g -JPh +ah'Ph,
\end{displaymath}
from which \eqref{P2g} follows.
\end{proof}

In view of \eqref{gk}, 
\begin{displaymath}
P=(\Gamma + \rho^{-1}kh')P(\Gamma +\rho^{-1}kh')'+kk',
\end{displaymath}
from which it follows that
\begin{displaymath}
\begin{split}
P-\Gamma P\Gamma' &= \rho^{-2}kk' +\rho^{-1}\Gamma Phk' +\rho^{-1}kh'P\Gamma'\\
&=(\Gamma Ph+\rho^{-1}k)(\Gamma Ph+\rho^{-1}k)' -\Gamma Phh'P\Gamma' ,
\end{split}
\end{displaymath}
which in turn yields
\begin{equation}
\label{CEEprel}
P = \Gamma (P-Phh'P) \Gamma' +gg'
\end{equation}
by  \eqref{gk}. By introducing interpolation data we shall deriver the appropriate CEE from  \eqref{CEEprel}.

\subsection{CEE for general interpolation data}\label{subsec:CEEdata}

We return to the general interpolation condition \eqref{interpolation}, where now
\begin{equation}
\label{ab2f}
f(z):=\phi_+(z^{-1})= \frac12\frac{b_{*}(z)}{a_{*}(z)},
\end{equation}
$a_{*}(z):=z^n a(z^{-1})$ being the reversed polynomial, and, in view of \eqref{covexpansion},
\begin{equation}
\label{covexpansion2}
f(z)=\tfrac{1}{2}+c_{1}z+c_{2}z^2+c_{3}z^3+\cdots .
\end{equation}
Given the interpolation values, we form the $(n+1)\times (n+1)$ matrix
\begin{subequations}\label{WWj}
\begin{equation}
\label{W}
W:=\begin{bmatrix}
W_{0}&~&~\\
~&\ddots&~\\
~&~&W_{m}
\end{bmatrix},
\end{equation}
where, for $j=0,1,\dots,m $,
\begin{equation}\label{Wj}
W_{j}=\begin{bmatrix}
w_{j0}&~&~&~\\
w_{j1}&w_{j0}&~&~\\
\vdots&\ddots&\ddots&~\\
w_{jn_{j-1}}&\cdots&w_{j1}&w_{j0}
\end{bmatrix} .
\end{equation}
\end{subequations}
In the same format we also define 
\begin{equation}\label{Z}
Z:=\begin{bmatrix}
Z_{0}&~&~\\
~&\ddots&~\\
~&~&Z_{m}
\end{bmatrix},\; Z_{j}=\begin{bmatrix}
z_{j}&~&~&~\\
1&z_{j}&~&~\\
~&\ddots&\ddots&~\\
~&~&1&z_{j}
\end{bmatrix}
\end{equation}
and the $n+1$-dimensional column vector
\begin{equation}\label{e}
e:=[e_{1}^{n_{0}},e_{1}^{n_{1}},\cdots,e_{1}^{n_{m}}]',
\end{equation}
where $e_{1}^{n_{j}}=[1,0,\cdots,0]\in\mathbb{R}^{n_j}$ for each $j=0,1,\dots,m$. 
Clearly $Z$ is a stability matrix with all eigenvalues in $\mathbb{D}$, and therefore the Lyapunov equation 
\begin{equation}\label{Q}
X=ZXZ^{*}+ee^{*} ,
\end{equation}
where $Z^*$ is the  Hermitian conjugate of $Z$, has a unique solution $X$ \cite[Proposition B.1.19]{LPbook}. The following result can, for example, be found in  \cite{b16,blomqvist,b1}.

\begin{proposition}\label{LyapunovW}
There exists a (strict) Carath\'eodory function $f$ satisfying \eqref{interpolation} if and only if 
\begin{equation}
\label{Pick}
\Sigma =WX+XW^*
\end{equation}
is positive definite.
\end{proposition}
The matrix $\Sigma$ is called the {\em generalized Pick matrix}.

In view of \eqref{covexpansion2},
\begin{equation}
\label{varphi(Z)}
f(Z)=\frac{1}{2}I+c_{1}Z+c_{2}Z^{2}+c_{3}Z^{3}\cdots=W
\end{equation}
\cite{Higham}, which together with $b_{*}(Z)=2f(Z)a_{*}(Z)$ yields 
\begin{equation*}
b_{*}(Z)e=2Wa_{*}(Z)e .
\end{equation*}
Therefore
$$V\begin{bmatrix}1\\b\end{bmatrix}=2WV\begin{bmatrix}1\\a\end{bmatrix}$$
where the matrix
\begin{equation}
\label{Vmatrix}
V:=[e,Ze,Z^{2}e,\cdots,Z^{n}e]
\end{equation}
is nonsingular by reachability. Therefore, by Lemma~\ref{lem:ab2g},
\begin{equation}\label{a2g}
\begin{bmatrix}0\\g\end{bmatrix}=(V^{-1}WV-\tfrac{1}{2}I)\begin{bmatrix}1\\a\end{bmatrix}=V^{-1}(W-\tfrac{1}{2}I)V\begin{bmatrix}1\\a\end{bmatrix},
\end{equation}
or equivalently
\begin{equation}
\label{a+g}
(W+\tfrac{1}{2}I)V\begin{bmatrix}0\\g\end{bmatrix}=(W-\tfrac{1}{2}I)V\begin{bmatrix}1\\a+g\end{bmatrix}.
\end{equation}
Since $(W+\tfrac{1}{2}I)$ is nonsingular, it follows from Lemma~\ref{lem:P2g} that 
\begin{equation}
\label{g2Psigma}
\begin{bmatrix}0\\g\end{bmatrix}=V^{-1}TV\begin{bmatrix}1\\\Gamma Ph+\sigma\end{bmatrix},
\end{equation}
where
\begin{equation}
\label{D}
T:=(W+\tfrac{1}{2}I)^{-1}(W-\tfrac{1}{2}I).
\end{equation}
Now defining the $n$-vector $u$ and the  $n\times n$-matrix $U$ via
\begin{equation}
\label{uU}
\begin{bmatrix}u&U\end{bmatrix}:=\begin{bmatrix}0&I_{n}\end{bmatrix}V^{-1}TV,
\end{equation}
where $I_n$ denotes the $n\times n$ identity matrix to distinguish it from the $(n+1)\times (n+1)$ identity matrix $I$, \eqref{g2Psigma} yields
\begin{equation} \label{g1}
g=u+U\sigma+U\Gamma Ph  ,
\end{equation}
which inserted into \eqref{CEEprel} yields precisely the Covariance Extension Equation
\begin{equation}\label{newCEE}
\begin{split}
P=&\Gamma(P-Phh'P)\Gamma'\\
&+(u+U\sigma+U\Gamma Ph)(u+U\sigma+U\Gamma Ph)' ,
\end{split}
\end{equation}
but now with $(u,U)$ exchanged for \eqref{uU}. Furthermore, by \eqref{gk}, \eqref{ab2g}, \eqref{Pk2g} and \eqref{rho},
\begin{equation}\label{P2ab}
\begin{split}
a&=(I-U)(\Gamma Ph+\sigma)-u\\
b&=(I+U)(\Gamma Ph+\sigma)+u\\
\rho&=\sqrt{1-h'Ph}
\end{split}
\end{equation}
in analogy with \eqref{abrho}.

\subsection{Main theorems}

Let the first column in \eqref{Wj} be denoted $(w_{j0}, w_j')'$ and form the $n$-vector 
\begin{equation}
\label{wdefn}
w=(w_0',w_{10}, w_1',w_{20}, w_2',\dots,w_{m0}, w_m')' ,
\end{equation}
where $w_{00}=\tfrac12$ has been removed since it is a constant and not a variable, and let 
$\mathcal{W}_+$ be the space of all $w$ such that $\Sigma$ in \eqref{Pick} is positive definite. Moreover, let $\mathcal{S}_{n}$ be the space of Schur polynomial of the form \eqref{sigma}.

The proof of the following proposition will be deferred to the appendix.

\begin{proposition}\label{w2uUprop}
There is map $u=\omega(w)$ sending $w$ to $u$, which is a diffeomorphism. Moreover, there is a linear  map $L$ such that $U=Lu$.
\end{proposition} 

The Covariance Extension Equation \eqref{newCEE} can be written
\begin{equation}
\label{LyapunovCEE}
P=\Gamma P\Gamma' + R(p), 
\end{equation}
where R(p) is a function of the first column $p:=Ph$ in the matrix variable $P$. Hence once $p$ has been determined, $P$ can be solved from the Lyapunov equation \eqref{LyapunovCEE}, since $\Gamma$ is a stability matrix. Consequently, CEE contains $n$ independent variables, the same number as the real dimension of $w$. 

Note that \eqref{newCEE} can reformulated as 
\begin{displaymath}
P=(\Gamma+\sigma h')P(\Gamma+\sigma h')'-(\Gamma Ph+\sigma)(\Gamma Ph+\sigma)' +gg' +\rho^2\sigma\sigma' ,
\end{displaymath}
where $g$ is given by \eqref{g1}. Since $g=\tfrac12(b-a)$ and $\Gamma Ph+\sigma=\tfrac12(a+b)$, this can rewritten as 
\begin{equation}
\label{ab2P}
P-JPJ'= -\frac12(ab' +ba')+\rho^2\sigma\sigma',
\end{equation}
where $a$ and $b$ are given by \eqref{P2ab}.

Let $\mathcal{P}_n$ be the $2n$-dimensional  space of pairs $(a,b)\in\mathcal{S}_{n}\times\mathcal{S}_{n}$ such that $f=b/a$ is a Carath\'eodory function. Moreover, for each $\sigma\in\mathcal{S}_{n}$, let 
$\mathcal{P}_n(\sigma)$ be the submanifold of $\mathcal{P}_n$ for which \eqref{absigma} holds. It was shown in \cite{b8} that  $\{\mathcal{P}_n(\sigma)\mid \sigma\in\mathcal{S}_{n}\}$ is a {\em foliation\/} of $\mathcal{P}_n$, i.e., a family of smooth nonintersecting submanifolds, called {\em leaves}, which together cover  $\mathcal{P}_n$. 

\begin{theorem}\label{difthm1}
Let $\sigma\in\mathcal{S}_n$. Then for each $w\in\mathcal{W}_+$ there is a unique $(a,b)\in\mathcal{P}_n(\sigma)$ such that \eqref{ab2f} satisfies the interpolation conditions \eqref{interpolation} and the positivity condition \eqref{absigma}. In fact, the map sending $(a,b)\in\mathcal{P}_n(\sigma)$ to $w\in\mathcal{W}_+$ is a diffeomorphism.
\end{theorem}

\begin{proof}
The Carath{\'e}odory function $f$ can be written
\begin{displaymath}
f(z)=\int_{-\pi}^\pi \frac{e^{i\theta}+z}{e^{i\theta}-z}\,\text{Re}\{\varphi(e^{i\theta})\}\frac{d\theta}{2\pi},
\end{displaymath}
where $(e^{i\theta}+z)(e^{i\theta}-z)^{-1}$ is a Herglotz kernel.
Hence the interpolation problem can be formulated as the generalized moment problem to find the Carath{\'e}odory function \eqref{ab2f} satisfying the moment conditions
\begin{equation}
\label{momentcond}
\int_{-\pi}^\pi \alpha_{jk}(e^{i\theta})\,\text{Re}\{\varphi(e^{i\theta})\}\frac{d\theta}{2\pi}=w_{jk},
\end{equation}
where
\begin{align*}
    \alpha_{j0}(z)&=\frac{z+z_j}{z-z_j}\quad j=0,1,\dots,m   \\
    \alpha_{jk}(z)&=\frac{2z}{(z-z_j)^{k+1}}\quad j=0,\dots,m, \, k=1,\dots, n_{j-1}  
\end{align*}
(see, e.g., \cite{BLkimura}). Then, by \cite[Theorem 3.4]{KLR}, there is a diffeomorphic map sending  $aa^*$ to $w$. However there is a smooth bijection between $aa^*$ and $a$, see, e.g., \cite[Section III]{BLGuM}. Given $\sigma$ and $a$, $b$ is uniquely determined via the linear relation  \eqref{absigma}. 
Note that $\rho^2$ is just the appropriate normalizing scalar factor once $(a,\sigma)$ has been chosen.
\end{proof}

\begin{theorem}\label{CEEthm}
For each $(\sigma,w)\in\mathcal{S}_n\times\mathcal{W}_+$, the Covariance Extension Equation  \eqref{newCEE} has a unique positive semidefinite solution $P$ with the property $h'Ph<1$, and \eqref{P2ab} is the corresponding unique solution of the analytic interpolation problem to find a rational Carath\'eodory function \eqref{ab2f} of degree at most $n$ satisfying the interpolation conditions \eqref{interpolation}. Moreover,
\begin{equation}
\label{rankP=degf}
\deg f =\rank\, P.
\end{equation}
\end{theorem}

\begin{proof} 
For each $(\sigma,w)\in\mathcal{S}_n\times\mathcal{W}_+$, by Theorem~\ref{difthm1}, there is a unique $(a,b)\in\mathcal{P}_n(\sigma)$, which means there is a unique rational positive real function $\phi_+(z)$ given by \eqref{ab2phi+}.
By the construction in Section~\ref{subsec:matrixrealization}, the algebraic Riccati equation \eqref{Riccati} has a unique minimal solution $P\geq 0$ satisfying \eqref{h'Ph<1}. By tranforming \eqref{Riccati} to \eqref{CEEprel} and inserting \eqref{g1}, there is a unique positive semidefinite solution to \eqref{newCEE} satisfying \eqref{h'Ph<1}. Relation \eqref{rankP=degf} follows from \eqref{degree=rank}.
\end{proof}

Finally we observe as in \cite{BFL} that $P$ can be eliminated from \eqref{ab2P} by  multiplying by $z^{j-i}$ and summing over all $i,j=1,2,\dots,n$, leading to an equation in merely the independent vector variable $p$. In fact, we recover \eqref{absigma}, which in matrix form can be written
\begin{equation}
\label{redequ}
S(a)\begin{bmatrix}1\\b\end{bmatrix}=2(1-h'p)\begin{bmatrix}s\\\sigma_n\end{bmatrix} ,
\end{equation}
where
\begin{equation*}
S(a)=
\begin{bmatrix}
1&\cdots&a_{n-1}&a_{n}\\
a_{1}&\cdots&a_{n}\\
\vdots&\ddots\\
a_{n}
\end{bmatrix}+\begin{bmatrix}
1&a_{1}&\cdots&a_{n}\\
~&1&\cdots&a_{n-1}\\
~&~&\ddots&\vdots\\
~&~&~&1
\end{bmatrix}
\end{equation*}
and
\begin{equation*}
s=\begin{bmatrix}
1+\sigma_{1}^{2}+\sigma_{2}^{2}+\cdots+\sigma_{n}^{2}\\
\sigma_{1}+\sigma_{1}\sigma_{2}+\cdots+\sigma_{n-1}\sigma_{n}\\
\vdots\\
\sigma_{n-1}+\sigma_{1}\sigma_{n}
\end{bmatrix},
\end{equation*}
where $a$ and $b$ are functions of $p$ via \eqref{P2ab}.
However, among the  $n+1$ equations \eqref{redequ}, the last is redundant \cite{BFL} and can be removed. Then we are left with $n$ equations
\begin{equation}\label{reducedCEE}
\begin{bmatrix}I_n&0\end{bmatrix}S(a)\begin{bmatrix}1\\b\end{bmatrix}=2(1-h'p)s
\end{equation}
in $n$ variables $p_1,p_2,\dots, p_n$.

\subsection{Back to rational covariance extension}

Next we show how the results presented in subsection~\ref{sec:CEE} follow from Theorem~\ref{CEEthm}. 
With $m=0$, $z_0=0$ and $w$ given by \eqref{wc} we have 
\begin{displaymath}
W=\begin{bmatrix}
\tfrac12&~&~&~\\
c_1&\tfrac12&~&~\\
\vdots&\ddots&\ddots&~\\
c_n&\cdots&c_1&\tfrac12
\end{bmatrix}
\end{displaymath}
and 
\begin{displaymath}
Z=\begin{bmatrix}
0&~&~&~\\
1&0&~&~\\
~&\ddots&\ddots&~\\
~&~&1&0\end{bmatrix},\quad
e=\begin{bmatrix}
1\\
0\\
\vdots\\
0
\end{bmatrix}.
\end{displaymath}
Since therefore $V=I$ and
\begin{displaymath}
D=\begin{bmatrix}1& \\c&C\end{bmatrix}^{-1}\begin{bmatrix}0& \\c&C-I_n\end{bmatrix}
=\begin{bmatrix}0&0 \\C^{-1}c&I_n-C^{-1}\end{bmatrix},
\end{displaymath}
where 
\begin{displaymath}
C=\begin{bmatrix}1&&&\\c_1&1&&\\\vdots&\vdots&\ddots&\\c_{n-1}&c_{n-2}&\cdots&1\end{bmatrix}, \quad 
c= \begin{bmatrix}c_1\\c_2\\\ \vdots\\c_n\end{bmatrix} ,
\end{displaymath} 
\eqref{uU} yields
\begin{equation}
\label{Cc}
u=C^{-1}c, \quad U=I_n-C^{-1}.
\end{equation}
To see that \eqref{Cc} is equivalent to  \eqref{uUconstr}, we identify negative powers of $z$ in 
\begin{displaymath}
(1+c_1z^{-1}+\dots +c_n z^{-n})(1-u_1z^{-1}-u_2z^{-2}-\dots)=1,
\end{displaymath}
 to obtain
\begin{displaymath}
c_k=u_k+\sum_{j=1}^{k-1}c_{k-j}u_j, \quad k=1,2,\dots,n,
\end{displaymath}
which is equivalent to 
\begin{equation}
\label{C2U}
Cu=c, \qquad C(I_n-U)=I_n .
\end{equation}

\subsection{An algorithm for solving CEE}\label{scalarhomotopy}

We shall use a homotopy continuation method to solve CEE, i.e., determine the unique positive semidefinite $P$ with the property $h'Ph<1$ that satisfies \eqref{newCEE} (Theorem~\ref{CEEthm}). For $u=0$, CEE takes the form
\begin{equation}
\label{CCEu=0}
P = \Gamma (P-Phh'P) \Gamma',
\end{equation}
which has the unique solution $P=0$. We would like to make a continuous deformation of $u$ to go between the solutions of \eqref{newCEE} and \eqref{CCEu=0}. To this end, we choose
\begin{equation}
\label{deformation}
u(\lambda)=\lambda u, \quad \lambda\in [0,1]. 
\end{equation}
Then $U(\lambda)=\lambda U$ (Proposition~\ref{w2uUprop}). Define $w(\lambda):=\omega^{-1}(\lambda u)$ in terms of the diffeomorphism in Proposition~\ref{w2uUprop} for all  $\lambda\in [0,1]$. It follows  from  \eqref{D} that $W=(I-T)^{-1}-\tfrac12 I$, and therefore the corresponding deformation is 
\begin{displaymath}
W(\lambda)=(I-\lambda T)^{-1}-\tfrac12 I . 
\end{displaymath}
We want to show that $W(\lambda)$ remains in $\mathcal{W}_+$ along the trajectory, i.e., that 
$W(\lambda)$ satisfies $\Sigma >0$ in \eqref{Pick} for all  $\lambda\in [0,1]$. To this end, a straightforward calculation yields 
\begin{align*}
\label{lambdaPick}
    \Sigma(\lambda)&:=W(\lambda)X+XW(\lambda)^*    \notag \\
    & =(I-\lambda T)^{-1}(X-\lambda^2 TXT^*)(I-\lambda T^*)^{-1} .
\end{align*}
However, $X-\lambda^2 TXT^* \geq X- TXT^* >0$ for all  $\lambda\in [0,1]$, and consequently $\Sigma(\lambda)>0$ as claimed.

To solve the reduced CEE in terms of $p=(p_1,p_2,\dots,p_n)$ we use the homotopy
\begin{equation}
\label{ }
\begin{split}
H(p,\lambda):= &\begin{bmatrix}I_n&0\end{bmatrix}S(a(p,\lambda))\begin{bmatrix}1\\b(p,\lambda)\end{bmatrix}\\
&-2(1-h'p)s =0
\end{split}
\end{equation}
where
\begin{subequations}\label{p2ab}
\begin{equation}
a(p,\lambda)=(I-\lambda U)(\Gamma p+\sigma)-\lambda u
\end{equation}
\begin{equation}
b(p,\lambda) =(I+\lambda U)(\Gamma p+\sigma)+\lambda u ,
\end{equation}
\end{subequations}
which also has a unique solution  $p(\lambda)$ for all $\lambda\in [0,1]$.

By the implicit function theorem we have the differential equation
\begin{equation}
\label{diffequ}
\frac{dp}{d\lambda}=\left[\frac{\partial H(p,\lambda)}{\partial p}\right]^{-1}\frac{\partial H(p,\lambda)}{\partial\lambda}, \quad p(0)=0,
\end{equation}
where
\begin{align*}
 \frac{\partial H(p,\lambda)}{\partial\lambda}  
   & =\begin{bmatrix}I_n&0\end{bmatrix}(S(a(p,\lambda))-S(b(p,\lambda)))\begin{bmatrix}0\\g(p,1)\end{bmatrix} \\
\frac{\partial H(p,\lambda)}{\partial p}  
&=\begin{bmatrix}I_n&0\end{bmatrix}(S(a(p,\lambda))+S(b(p,\lambda)))\begin{bmatrix}0\\\Gamma\end{bmatrix}+2hs'\\
&+\begin{bmatrix}I_n&0\end{bmatrix}(S(a(p,\lambda))-S(b(p,\lambda)))\begin{bmatrix}0\\ \lambda U\Gamma\end{bmatrix}
\end{align*}
and
\begin{equation}\label{g(p)} 
g(p,\lambda)= u(\lambda) +U(\lambda)\sigma + U(\lambda)\Gamma p. 
\end{equation}
The differential equation \eqref{diffequ} has a unique solution $p(\lambda)$ on the interval $\lambda\in[0,1]$, so by solving  
the Lyapunov equation
\begin{equation}
\begin{split}
&P-\Gamma P\Gamma'=-\Gamma p(1)p(1)'\Gamma'+\\
&\qquad (u+U\sigma+U\Gamma p(1))(u+U\sigma+U\Gamma p(1))' ,
\end{split}
\end{equation}
we obtain the unique solution of \eqref{newCEE} \cite[Proposition B.1.19]{LPbook}. To solve the differential equation \eqref{diffequ} we use  predictor-corrector steps \cite{AllgowerGeorg}. 

\section{Multivariable analytic interpolation}\label{sec:multivariable}

Next we consider the multivariable version of the problem stated in Section~\ref{sec:intro}. More precisely, let $F$ be an $\ell\times \ell$ matrix-valued real rational function, analytic in the unit disc $\mathbb{D}$, which satisfies the interpolation condition 
\begin{align}
\label{multinterpolation}
 \frac{1}{k!}F^{(k)}(z_{j})=W_{jk},\quad& j=0,1,\cdots,m,   \\
    &   k=0,\cdots n_{j}-1 ,\notag
\end{align}
and the positivity condition 
\begin{equation}
\label{F+F*}
F(e^{i\theta}) + F(e^{-i\theta})' > 0, \quad -\pi\leq \theta\leq\pi .
\end{equation}
We restrict the complexity of the rational function $F(z)$ by requiring that its McMillan degree be at most $\ell n$, where
\begin{equation}
\label{deg(f)}
n=\sum_{j=0}^{m}n_j -1 .
\end{equation}
Without loss of generality we may assume that $ z_{0}=0$ and $W_{0}=\frac{1}{2}I$.
Then $F(z)$  has a realization 
 \begin{equation}
\label{ }
F(z)=\tfrac12 I + zH(I-zF)^{-1}G,
\end{equation}
where $H\in\mathbb{R}^{\ell\times\ell n}$, $F\in\mathbb{R}^{\ell n\times\ell n}$, $G\in\mathbb{R}^{\ell n\times\ell}$, $(H,F)$ is an observable pair, and the matrix $F$ has all its eigenvalues in $\mathbb{D}$. 

In analogy with the construction in subsection~\ref{subsec:CEEdata} we form the $\ell(n+1)\times\ell(n+1)$ matrix
\begin{equation}
\label{W}
W:=\begin{bmatrix}
W_{0}&~&~\\
~&\ddots&~\\
~&~&W_{m}
\end{bmatrix}
\end{equation}
with
\begin{equation}\label{Wj2}
W_{j}=\begin{bmatrix}
W_{j0}&~&~&~\\
W_{j1}&W_{j0}&~&~\\
\vdots&\ddots&\ddots&~\\
W_{jn_{j-1}}&\cdots&W_{j1}&W_{j0}
\end{bmatrix}
\in \mathbb{C}^{\ell  n_j \times \ell  n_j}
\end{equation}
for each $j=0,1,\dots,m$.
Let $X$ be the unique solution of the Lyapunov equation \eqref{Q}. The inverse problem to determine the interpolant $F(z)$ has a solution if and only if the Pick-type condition 
\begin{equation}
\label{Pickmatrix}
W(X\otimes I_\ell) + (X\otimes I_\ell)W^* > 0 ,
\end{equation} 
is satisfied, where $\otimes$ denotes Kronecker product.

\subsection{Multivariable stochastic realization theory}\label{subsec:realization}

Following the pattern in subsection~\ref{subsec:realization} we define 
\begin{equation}
\label{Phi+}
\Phi_+(z):=F(z^{-1})=\tfrac12 I + H(zI-F)^{-1}G,
\end{equation}
which is  (strictly) positive real. Moreover, define the minumum-phase spectral factor $V(z)$ satisfying
\begin{equation}
\label{VPhi+}
V(z)V(z^{-1})'=\Phi(z) := \Phi_+(z) + \Phi_+(z^{-1})' ,
\end{equation}
which then has a realization of the form
\begin{equation}
\label{V(z)}
V(z)=H(zI-F)^{-1}K + R
\end{equation}
\cite[Chapter 6]{LPbook}. Now, by the usual coordinate transformation $(H,F,G)\to(HT^{-1},TFT^{-1},TG)$ we can choose $(H,F)$ in the observer canonical form
\begin{displaymath}
H=\text{diag}(h_{t_1},h_{t_2},\dots,h_{t_\ell}) \in \mathbb{R}^{\ell\times n\ell}
\end{displaymath}
with $h_\nu:=(1,0,\dots,0)\in\mathbb{R}^\nu$, and
\begin{equation}
\label{F}
F=J-AH \in\mathbb{R}^{n\ell\times n\ell}
\end{equation}
where $J:=\text{diag}(J_{t_1},J_{t_2},\dots, J_{t_\ell})$ with $J_\nu$ the $\nu\times\nu$ shift matrix
\begin{displaymath}
J_\nu =\begin{bmatrix}0&1&0&\dots&0\\0&0&1&\dots&0\\\vdots&\vdots&\vdots&\ddots&0\\
0&0&0&\dots&1\\0&0&0&\dots&0\end{bmatrix}
\end{displaymath}
and $A\in\mathbb{R}^{n\ell\times \ell}$.
The numbers $t_1,t_2,\dots,t_\ell$ are the {\em observability indices\/} of $\Phi_+(z)$, and 
\begin{equation}
\label{tsum}
t_1+t_2+\dots+t_\ell=n\ell.
\end{equation}
Moreover, define 
\begin{equation}
\label{PI}
\Pi(z):=\text{diag}(\pi_{t_1}(z),\pi_{t_2}(z),\dots,\pi_{t_\ell}(z)),
\end{equation}
where $\pi_\nu(z)=(z^{\nu-1},\dots,z,1)$ ,
\begin{equation}
\label{D(z)}
D(z):=\text{diag}(z^{t_1},z^{t_2},\dots, z^{t_\ell}).
\end{equation}
and 
\begin{equation}
\label{A(z)}
A(z)=D(z) +\Pi(z)A.
\end{equation}

\begin{lemma}
The rational matrix functions \eqref{Phi+} and \eqref{V(z)} have the matrix fraction representations
\begin{subequations} \label{AinvB}
\begin{equation}
 \Phi_+(z) =\tfrac12 A(z)^{-1}B(z) ,
\end{equation}   
where
\begin{equation}
\label{B(z)}
B(z)=D(z) +\Pi(z)B\quad \text{with $B=A+2G$}
\end{equation}
\end{subequations}
and
\begin{subequations} \label{AinvSigma}
\begin{equation}
V(z)  =A(z)^{-1}\Sigma(z)R,
\end{equation} 
where
\begin{equation}
\label{Sigma(z)}
\Sigma(z)=D(z)+\Pi(z)\Sigma \quad \text{with $\Sigma = A+KR^{-1}$}.
\end{equation}
\end{subequations}
\end{lemma}

\begin{proof}
Since $\Pi(z)(zI-J)=D(z)H$,
\begin{displaymath}
\Pi(z)(zI-F)=\Pi(z)(zI-J)+\Pi(z)AH=A(z)H,
\end{displaymath}
and hence 
\begin{equation}
\label{Ainvlem}
H(zI-F)^{-1}=A(z)^{-1}\Pi(z).
\end{equation}
Then \eqref{AinvB} and \eqref{AinvSigma} follow from \eqref{Phi+} and \eqref{V(z)}, respectively.
\end{proof}

It follows from stochastic realization theory \cite[Chapter 6]{LPbook} that 
\begin{subequations}\label{KR}
\begin{align}
  K  & =(G-FPH')(R')^{-1}  \label{K}\\
  RR'  &  = I-HPH' \label{R}
\end{align}
\end{subequations}
where $P$ is the minimal symmetric solution of the algebraic Riccati equation
\begin{equation}
\label{Riccati2}
P=FPF' + (G-FPH')(I-HPH')^{-1}(G-FPH')'.
\end{equation}
Then, from \eqref{F} and \eqref{KR} we have
\begin{align*}
   G &= JPH'  -AHPH' +KR^{-1}(I-HPH') \\
      & = \Gamma PH' +KR^{-1},
\end{align*}
where, by \eqref{Sigma(z)},
\begin{equation}
\label{Gamma}
\Gamma=J-\Sigma H,
\end{equation}
and consequently
\begin{equation}
\label{G}
G=\Gamma PH'+\Sigma -A.
\end{equation}

We are now in a position to derive the multivariable version of \eqref{CEEprel}, namely
\begin{equation}
\label{AREmod}
P=\Gamma (P-PH'HP)\Gamma' +GG' .
\end{equation}
In fact, noting that $F=\Gamma +KR^{-1}H$ and $G-\Gamma PH'=KR^{-1}$,  we see that \eqref{Riccati2} can be written
\begin{align*}
   P & = (\Gamma +KR^{-1}H)P(\Gamma +KR^{-1}H)' +KK'\\
    &  = \Gamma P\Gamma' +\Gamma PH'(KR^{-1})' +KR^{-1}HP\Gamma' \\
    & \phantom{xxxxxxxxxxxxxxxxxxxx} +KR^{-1}(KR^{-1})' ,
\end{align*}
where we have also used \eqref{R}. Then inserting $KR^{-1}=G-\Gamma PH'$ we obtain \eqref{AREmod}.

\subsection{The multivariable Covariance Extension  Equation}\label{sec:multiCEE}

Next we introduce the interpolation condition \eqref{multinterpolation}.

\begin{lemma}
The interpolation condition \eqref{multinterpolation} can be written
\begin{equation}
\label{multinterpolation2}
F(Z\otimes I_\ell)=W,
\end{equation}
where the matrices $W$ and $Z$ are given by \eqref{W} and \eqref{Z}, respectively. 
\end{lemma}

\begin{proof}
Since $F(z)$ is analytic in $\mathbb{D}$, it has a representation 
\begin{displaymath}
F(z)=\sum_{k=0}^\infty C_k z^k
\end{displaymath}
there, where $C_0=\frac12 I_\ell$. A straight-forward calculation yields
\begin{displaymath}
F(Z_j\otimes I_\ell)= \sum_{k=0}^\infty (Z_j)^k\otimes C_k =W_j,
\end{displaymath}
where $W_j$ is given by \eqref{Wj2}. 
Then \eqref{multinterpolation2} follows from \eqref{Z} and \eqref{W}.
\end{proof}

Analogously to the situation in subsection~\ref{subsec:CEEdata},  \eqref{AinvB} provides us with the representation
\begin{subequations}\label{AB2F}
\begin{equation}
F(z)=\tfrac{1}{2}A_*(z)^{-1}B_*(z)
\end{equation}
in terms of the reversed matrix polynomials
\begin{align}
A_*(z)&=D(z)A(z^{-1})=I_\ell +D(z)\Pi(z^{-1})A \label{Astar}\\
B_*(z)&=D(z)B(z^{-1})=I_\ell +D(z)\Pi(z^{-1})B \label{Bstar} ,
\end{align}
\end{subequations}
where $D(z)$ is given by \eqref{D(z)}. Then  the interpolation condition  \eqref{multinterpolation2} takes the form
\begin{equation}
\label{multinterpolation3}
2A_*(Z\otimes I_\ell)W=B_*(Z\otimes I_\ell).
\end{equation}
In view of \eqref{A(z)} and \eqref{B(z)} we have the polynomial representations
\begin{subequations}
\begin{align}
A_*(z)&=I_\ell + A_1z+A_2z^2+\dots+A_tz^t \label{A*}\\
B_*(z)&=I_\ell + B_1z+B_2z^2+\dots+B_tz^t \label{B*} ,
\end{align}
\end{subequations}
where $t$ is the largest observability index. Introducing $Q:=A+G$, it follows from \eqref{B(z)} that $A=Q-G$ and $B=Q+G$,
so the interpolation condition \eqref{multinterpolation3} can be written'
\begin{equation}
\label{multinterpolation4}
G_*(Z\otimes I_\ell)=Q_*(Z\otimes I_\ell)T,
\end{equation}
where
\begin{align}
G_*(z)&= G_1z+G_2z^2+\dots+G_tz^t \label{A*}\\
Q_*(z)&=I_\ell + Q_1z+Q_2z^2+\dots+Q_tz^t \label{B*} 
\end{align}
and
\begin{equation}
\label{T}
T:=(W-\tfrac{1}{2}I)(W+\tfrac{1}{2}I)^{-1}
=\begin{bmatrix}
T_{0}&~&~\\
~&\ddots&~\\
~&~&T_{m}
\end{bmatrix},
\end{equation}
where
\begin{equation}\label{Tj}
T_{j}=\begin{bmatrix}
T_{j0}&~&~&~\\
T_{j1}&T_{j0}&~&~\\
\vdots&\ddots&\ddots&~\\
T_{jn_{j-1}}&\cdots&T_{j1}&T_{j0}
\end{bmatrix}
\end{equation}
for $j=0,1,\dots,m $.
Now, \eqref{multinterpolation4} yields 
\begin{align*}
    &Z\otimes  G_1 +Z^2\otimes G_2 +  \dots + Z^t\otimes G_t   \\
    &= (I_{\ell(n+1)}+ Z\otimes  Q_1 +Z^2\otimes Q_2 +  \dots + Z^t\otimes Q_t)T.
\end{align*}
Multiplying both sides from the right by $(e\otimes I_{\ell})$ and observing that, in view of the rule
\begin{equation}
\label{rule}
(A\otimes B)(C\otimes D)=(AC)\otimes (BD),
\end{equation}
which holds for arbitrary matrices of appropriate dimensions,
\begin{displaymath}
(Z^k\otimes G_k)(e\otimes I_\ell)=(Z^k e)\otimes  G_k= (Z^k e\otimes I_\ell)G_k,
\end{displaymath}
we have
\begin{equation}
\label{multinterpolation6}
V \begin{bmatrix} G_1\\  \vdots \\G_t \end{bmatrix}=\hat{T}+(Z\otimes  Q_1 +Z^2\otimes Q_2 +  \dots + Z^t\otimes Q_t)\hat{T},
\end{equation}
where $V$ is the $\ell(n+1)\times\ell t$ matrix
\begin{equation}
\label{V}
V:=\begin{bmatrix}Ze\otimes I_{\ell}&\cdots&(Z^{t}e)\otimes I_{\ell}\end{bmatrix} 
\end{equation}
and $\hat{T}$ is the $\ell(n+1)\times\ell$ matrix
\begin{equation}
\label{That}
\hat{T} :=T(e\otimes I_{\ell}).
\end{equation}
Here
\begin{equation}
\label{That2}
\hat{T}=\begin{bmatrix}
\hat{T}_{0}\\\hat{T}_{1}\\\vdots\\\hat{T}_{m}
\end{bmatrix},\quad\text{where}\;
\hat{T}_{j}=
\begin{bmatrix}
T_{j0}\\
T_{j1}\\
\vdots\\
T_{jn_{j}-1}
\end{bmatrix} .
\end{equation}

Next let $N_1, N_2, \dots, N_t$ be the $\ell\times\ell n$ matrices defined by 
\begin{equation}
\label{N}
D(z)\Pi(z^{-1})=N_1z+N_2z^2+\dots +N_tz^t.
\end{equation}
Then $A_j=N_j A$, $B_j=N_j B$, $G_j=N_j G$ and $Q_j=N_j Q$ for $j=1,2,\dots,t$, and therefore \eqref{multinterpolation6} takes the form 
\begin{equation}
\label{multinterpolation7}
VNG=\hat{T}+(Z\otimes  N_1Q + \dots + Z^t\otimes N_tQ)\hat{T},
\end{equation}
where
\begin{equation}\label{Nmatrix}
N=\begin{bmatrix}N_{1}\\\vdots\\N_{t}\end{bmatrix}\in\mathbb{R}^{\ell t\times \ell n}, 
\; N_{k}=\begin{bmatrix}e_{t_1}^{k}\\~&e_{t_2}^{k}\\~&~&\ddots\\~&~&~&e_{t_\ell}^{k}\end{bmatrix}
	\end{equation}
Here $e_{j}^{k} $ is a $1\times j$ row vector with the k:th element  being 1 and the others 0 whenever $k\leq j$, and a zero row vector of dimension $1\times j$ when $k>j$. Now, $VN$ is an $\ell(n+1)\times\ell n$ matrix in which the top $\ell$ rows are zero, since $z_0=0$, i.e., it takes the form
\begin{equation}
\label{VN}
VN=\begin{bmatrix} 0_{\ell\times\ell n}\\L\end{bmatrix}.
\end{equation}
To derive the multivariable CEE we would like to solve \eqref{multinterpolation7} for $G$ and insert it in \eqref{AREmod}. This would be possible if the square matrix $L$ is nonsingular, in which case  $VN$ would have a psuedo-inverse $(VN)^\dagger$.

\begin{lemma}\label{VNlem}
The $\ell n\times \ell n$ matrix $L$ defined by \eqref{VN} is nonsingular if and only if all observability indices are the same, i.e., $t_1=t_2=\dots=t_\ell =n$.
\end{lemma}

\begin{proof}
Ordering the the observability indices as $$t_1\geq t_2\geq \dots\geq t_\ell$$ and setting $t:=t_1$, we have $t\geq n$ by \eqref{tsum}. Since $(Z,e)$ is a reachable pair, 
\begin{equation}
\label{rank=n}
\text{rank\,} \begin{bmatrix}Ze&Z^{2}e&\cdots&Z^t e\end{bmatrix} =n.
\end{equation}

First assume that $t=n$. Then, since $\text{rank}(A\otimes B)=\text{rank}(A)\text{rank}(B)$, 
\begin{displaymath}
V=\begin{bmatrix}Ze&Z^{2}e&\cdots&Z^n e\end{bmatrix}\otimes I_\ell \in\mathbb{C}^{\ell n\times \ell n}
\end{displaymath}
has rank $\ell n$, and so does $N\in\mathbb{R}^{n\ell\times n\ell}$.  Therefore Sylverster's inequality,
\begin{displaymath}
\text{rank\,}V+\text{rank\,}N-\ell n\leq \text{rank\,} VN \leq \text{min\,}(\text{rank\,}V,\text{rank\,}N),
\end{displaymath}
(see, e.g., \cite[p.741]{LPbook}) implies that $VN$ has rank $\ell n$, and hence $L$ is nonsingular. 

Next assume that $t>n$. Then the first $t$ columns of $N$ can be written $I_t \otimes (e^1_\ell)'$, so the first $t$ columns of $VN$ form the matrix
\begin{displaymath} 
\begin{split}
\left(\begin{bmatrix}Ze&Z^{2}e&\cdots&Z^t e\end{bmatrix}\otimes I_\ell\right)\left(I_t \otimes(e^1_\ell)'\right)\\
=\begin{bmatrix}Ze&Z^{2}e&\cdots&Z^t e\end{bmatrix}\otimes (e^1_\ell)',
\end{split}
\end{displaymath}
which in view of \eqref{rank=n} has rank $n<t$. Hence the columns of $VN$ are linearly dependent, and thus $L$ is singular. 
\end{proof}

Consequently, assuming that all observability indices are the same, we can solve \eqref{multinterpolation7} for $G$ to obtain
\begin{equation}
\label{multiinterpolation8}
G =(VN)^\dagger\hat{T} +(VN)^\dagger(Z\otimes  N_1Q + \dots + Z^t\otimes N_tQ)\hat{T}
\end{equation}
Since $Q=A+G$, \eqref{G} yields
\begin{equation}
\label{GammaSigma2G}
G=u + U(\Gamma PH'+\Sigma),
\end{equation}
where $u:=(VN)^\dagger\hat{T}$ and $U: \mathbb{R}^{\ell n\times\ell}\to\mathbb{R}^{\ell n\times\ell}$ is the linear operator 
\begin{displaymath}
Q\mapsto (VN)^\dagger(Z\otimes  N_1Q + \dots + Z^t\otimes N_tQ)\hat{T}.
\end{displaymath}
Then, inserting \eqref{GammaSigma2G} into \eqref{AREmod} we obtain the multivariable Covariance Extension Equation
\begin{equation}
\begin{split}
\label{multCEE}
P =&\Gamma (P-PH'HP)\Gamma' \\
&+(u + U\Sigma +U\Gamma PH')(u + U\Sigma +U\Gamma PH')' .
\end{split}
\end{equation}

\subsection{Main results in the multivariable case}

We redefine $\mathcal{S}_n$ for the multivarible case  to be the class of $\ell\times\ell$ matrix polynomials \eqref{A(z)} such that $\det A(z)$ has all its zeros in the open unit disc $\mathbb{D}$. Moreover, let $\mathcal{W}_+$ be the values in \eqref{multinterpolation} that satisfy the generalized Pick condition \eqref{Pickmatrix}. In the present matrix case, the relation \eqref{absigma} becomes
\begin{equation}
\label{AB2Sigma}
A(z)B(z^{-1})'+B(z)A(z^{-1})'=2\Sigma(z)RR'\Sigma(z^{-1})'.
\end{equation}
Let $\mathcal{P}_n$ be the space of pairs $(A,B)\in\mathcal{S}_{n}\times\mathcal{S}_{n}$ such that $A(z)^{-1}B(z)$ is positive real. 
Then the  problem at hand is to find, for each $\Sigma\in\mathcal{S}_{n}$, a pair $(A,B)\in\mathcal{P}_n$ such that  \eqref{AB2Sigma} and \eqref{multinterpolation} hold. 

Clearly $\mathcal{S}_n$ consists of subclasses with different Jordan structures $J$ defined via \eqref{F}. In each such subclass $D(z)$ and $\Pi(z)$ in \eqref{A(z)}, as well as $N_1,N_2,\dots,N_t$ in \eqref{N}, are the same.

From this calculation we have the following theorem. For the details of the proof of the last statement \eqref{deg=rank} we refer to \cite{BLpartial}.

\begin{theorem}\label{mainthm}
Given $(\Sigma,W)\in\mathcal{S}_n\times\mathcal{W}_+$, where $\Sigma(z)$ has all it observability indices equal.  Then there is a positive semidefinite solution $P$ to the Covariance Extension Equation \eqref{multCEE} such that $HPH'< I$.  To any such $P$ there corresponds a unique analytic interpolant \eqref{AB2F}, where  the matrices $A$ and $B$ are given by 
\begin{equation}
\label{AB}
\begin{split}
A&=(I-U)(\Gamma PH'+\Sigma)-u \\
B&=(I+U)(\Gamma PH'+\Sigma)+u
\end{split}
\end{equation}
The matrix polynomials $A(z)$ and $B(z)$ have the same Jordan structure as $\Sigma(z)$,  and they satisfy \eqref{AB2Sigma} with 
\begin{equation}
\label{P2R}
R=(I-HPH')^{\frac{1}{2}}.
\end{equation}
Finally, 
\begin{equation}
\label{deg=rank}    
\deg F = \rank\, P.
\end{equation}
\end{theorem}

 This result is considerably weaker than the scalar version  Theorem~\ref{CEEthm}. Theorem~\ref{mainthm} does not guarantee that a solution to \eqref{multCEE} is unique. In fact, if there were two solutions to \eqref{multCEE}, there would be two interpolants, a unique one for each solution $P$.
Moreover, the condition on the observability indices  restricts the classes of Jordan structures that are feasible.  

\begin{theorem}\label{mainthm2}
Given $(\Sigma,W)\in\mathcal{S}_n\times\mathcal{W}_+$, where $\Sigma(z)=\sigma(z) I$ with $\sigma(z)$  a scalar Schur polynomial.  Then there is a unique positive semidefinite solution $P$ to the Covariance Extension Equation \eqref{multCEE} such that $HPH'< I$ and a corresponding  unique analytic interpolant \eqref{AB2F}, where $A(z)$ and $B(z)$ have the same Jordan structure as $\Sigma(z)$, and the matrices $A$ and $B$ are obtained as in  Theorem~\ref{mainthm}.
Finally, $\deg F(z) = \rank P$.
\end{theorem}

The observability indices of $\Sigma(z)$ in Theorem~\ref{mainthm2} are all the same. Moreover, for this case, existence and uniqueness of the underlying multivariable analytic interpolation problem have already been established \cite{BLN,FPZbyrneslindquist}. Then the proof of Theorems~\ref{difthm1} and \ref{CEEthm} can be modified for the resulting setting {\em mutatis mutandis}. 

Recently there have been several results \cite{FPZbyrneslindquist,RFP,Takyar,ZhuBaggio,Zhu} on the question of existence and uniqueness of the multivariate analytic interpolation problem, mostly for the covariance extension problem ($m=0, n_0=n+1$), but there are so far only partial results and for special structures of the prior (in our case $\Sigma(z)$). Especially the question of uniqueness has proven elusive. Perhaps, as suggested in \cite{Takyar}, this is due to the Jordan structure, and this could be the reason for the condition on the observability indices required in Theorem~\ref{mainthm}. In any case, as long as our algorithm delivers a solution to the Covariance Extension Equation, we will have a solution to the analytic interpolation problem, unique or not. An advantage of our method is that  \eqref{deg=rank} can be used for model reduction, as will be illustrated in Section~\ref{sec:examples}.

\subsection{An algorithm for the multivariable CEE}\label{sec:multialgorithm}

As in the scalar case we shall use a homotopy continuation method. We assume from now on that $t:= t_1=t_2=\dots,t_\ell =n$. When $u=0$, $\hat{T}=0$, and hence $U=0$. Then the modified Riccati equation \eqref{multCEE} becomes $P=\Gamma(P-PH'HP)\Gamma'$, which has the solution $P=0$. We would like to make a continuous deformation of $u$ to go from this trivial solution to the solution of \eqref{multCEE}, so we choose $u(\lambda)=\lambda u$ with $\lambda\in [0,1]$. The corresponding deformation of $U$ is $\lambda U$, and  $T$ is deformed to $\lambda T$. Since \eqref{T} implies that $W=(I- T)^{-1} -\tfrac12 I$, the value matrix \eqref{W} will then vary as $W(\lambda)=(I-\lambda T)^{-1} -\tfrac12 I$. Then, the proof that $W(\lambda)\in\cal{W}_+$ is {\em mutatis mutandis\/} the same as in subsection~\ref{scalarhomotopy}. Hence $W(\lambda)$ satisfies \eqref{Pickmatrix}
 along the whole trajectory.
 
 Analogously with the scalar case, we reduce the problem to solving for the $n\ell\times \ell$ matrix
\begin{equation}
p=PH' .
\end{equation}
To this end, we note that the matrix version of \eqref{AB2Sigma} is 
\begin{equation}
\label{SM}
\begin{split}
&S(A)M(B)+S(B)M(A)\\
&=2S(\Sigma)(I_{n+1}\otimes RR')M(\Sigma)
\end{split}
\end{equation}
where
\begin{equation*}
S(A)=\begin{bmatrix}
I&A_{1}&\cdots&A_n\\
~&I&\cdots&A_{n-1}\\
~&~&\ddots&\vdots\\
~&~&~&I
\end{bmatrix}\qquad M(A)=\begin{bmatrix}
I\\
A_{1}'\\
\vdots\\
A_{n}'
\end{bmatrix}.
\end{equation*}

From \eqref{AB}, \eqref{R}  and \eqref{G} we have 
\begin{displaymath}
A+B=2(\Gamma PH'+\Sigma)=2(JPH' +\Sigma RR').
\end{displaymath}
Since $e^n_nJ_n=0$ and hence $N_nJ=0$, this yields the relation 
\begin{displaymath}
A_n+B_n=2\Sigma_nRR' 
\end{displaymath}
between  $A_n=N_nA$, $B_n=N_nB$ and $\Sigma_n=N_n\Sigma$. This is the same as the last block row in \eqref{SM}, which can therefore be deleted, leaving us with
\begin{align*}
\begin{bmatrix}
I_{n\ell}&0_{n\ell\times \ell}
\end{bmatrix}&(S(A)M(B)+S(B)M(A))\\
&=2\begin{bmatrix}I_{n\ell}&0_{n\ell\times \ell}\end{bmatrix}S(\Sigma)(I_{n+1}\otimes RR')M(\Sigma) .
\end{align*}

Consequently, we use the homotopy 
\begin{equation}
\begin{split}
\mathcal{H}(p,\lambda):=&\begin{bmatrix}
I_{n\ell}&0_{n\ell\times \ell}
\end{bmatrix}\big(S(A)M(B)+S(B)M(A)\\
&-2S(\Sigma)(I_{n+1}\otimes (I-Hp))M(\Sigma)\big)=0 ,
\end{split}
\end{equation}
where
\begin{equation}
\begin{split}
A&=A(p,\lambda):=\Gamma p+\Sigma-\lambda u-\lambda U(\Gamma p +\Sigma) \\
B&=B(p,\lambda):=\Gamma p+\Sigma+\lambda u+\lambda U(\Gamma p +\Sigma)\\
\end{split}
\end{equation}
depend on $(p,\lambda)$, thus reducing the problem to solving the differential equation
\begin{equation}
\begin{split}
\label{diffequ2}
&\frac{d}{d\lambda}\text{vec}(p(\lambda))=\left[\frac{\partial \text{vec}(\mathcal{H}(p,\lambda))}{\partial \text{vec}(p)}\right]^{-1}\frac{\partial \text{vec}(\mathcal{H}(p,\lambda))}{\partial\lambda}\\
&\text{vec}(p(0))=0
\end{split}
\end{equation}
\cite{AllgowerGeorg}, which has the solution $ \hat{p}(\lambda) $ for $ 0\leq\lambda\leq1$. The solution of \eqref{multCEE} is then obtained by finding the unique solution of the  Lyapunov equation
\begin{equation}
\begin{split}
&P-\Gamma P\Gamma' =-\Gamma p(1)p(1)'\Gamma' \\
&\phantom{xxxxx}+(u + U(\Gamma p(1) +\Sigma))(u + U(\Gamma p(1) +\Sigma))'.
\end{split}
\end{equation}

\section{Some numerical examples}\label{sec:examples}

\subsection{Spectral estimation with model reduction}

Consider a transfer function \eqref{v}, i.e.,
\begin{displaymath}
v(z)=\rho\frac{\sigma(z)}{a(z)},
\end{displaymath}
of degree 7 with zeros at $0.9e^{\pm2.6i}$, $0.5e^{\pm1.3i}$, $0.94e^{\pm1.6i}$, $0.3$, poles at $0.1e^{\pm1.9i}$, $0.8e^{\pm1.35i}$, $0.7e^{\pm2.1i}$, $0.1$, and $\rho=0.5$, as depicted in Fig.~\ref{fig2}.
\begin{figure}[!htp]
	\centering
	\includegraphics[width=0.30\textwidth]{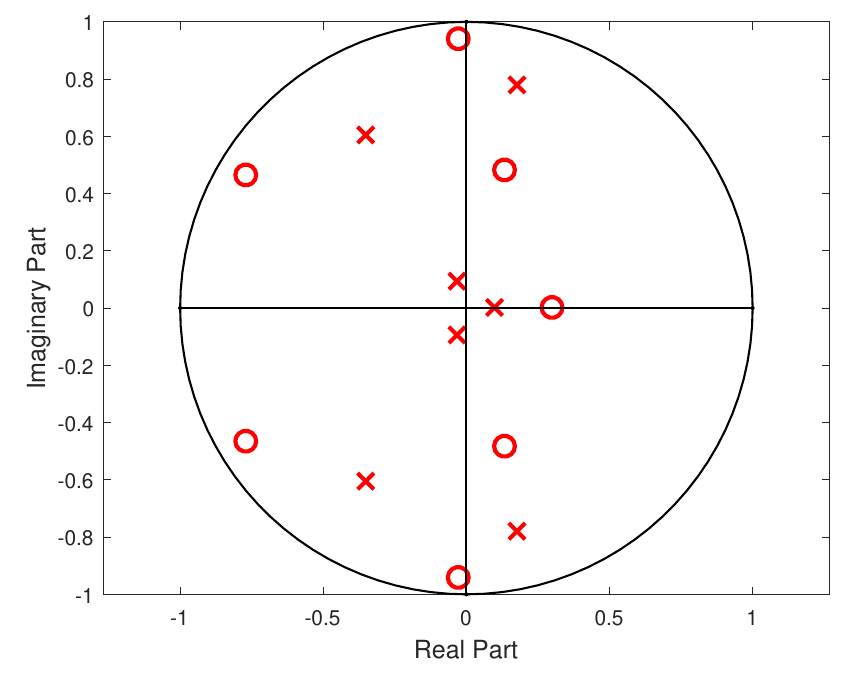}
	\caption{The position of poles ($\times$) and zeros ($\circ$).}
	\label{fig2}
\end{figure}
Now passing normalized white noise through the filter 
\[
\text{white noise\;\,}\negmedspace{\longrightarrow}\fbox{$v(z)$}\negmedspace\stackrel{y}
{\longrightarrow} 
\]
with $v(z)$ as its transfer function, we generate an observed time series $y_0,y_1,y_2,\dots,y_N$. Then, using the covariance estimates \eqref{ckest} and appropriate choice of $\sigma(z)$, we can determine an estimate of the spectral density of $y$ by using the methods in subsection~\ref{sec:CEE}.  In terms of the general interpolation problem \eqref{interpolation} this corresponds to choosing $m=0$, $n_0=8$ and $z_0=0$, i.e., all the interpolation points at zero. 

However it was shown in \cite{b2,GL1} that a higher resolution estimate (in a designated band of frequences) can be obtained by moving some of the interpolation points away from zero closer to the unit circle. This in known by the acronym THREE (Tunable High REeolution Estimator). This can be one by passing the signal $y$ through a bank of filters as in Fig.~\ref{fig:box1} with
$$G_{j}(z)=z(zI-Z_{j})^{-1}e_{1}^{n_{j}},\quad j=0,1,\cdots,m ,$$
where $Z_j$ is given by \eqref{Z}. In the present example we choose $n_0=4$, $n_1=n_2=n_3=n_4=1$, $z_1=0.98e^{2.1i}$, $z_2=0.98e^{-2.1i}$, $z_3=0.99$ and $z_4=-0.99$. 
\begin{figure}[!htp]
	\centering
	\includegraphics[width=0.32\textwidth]{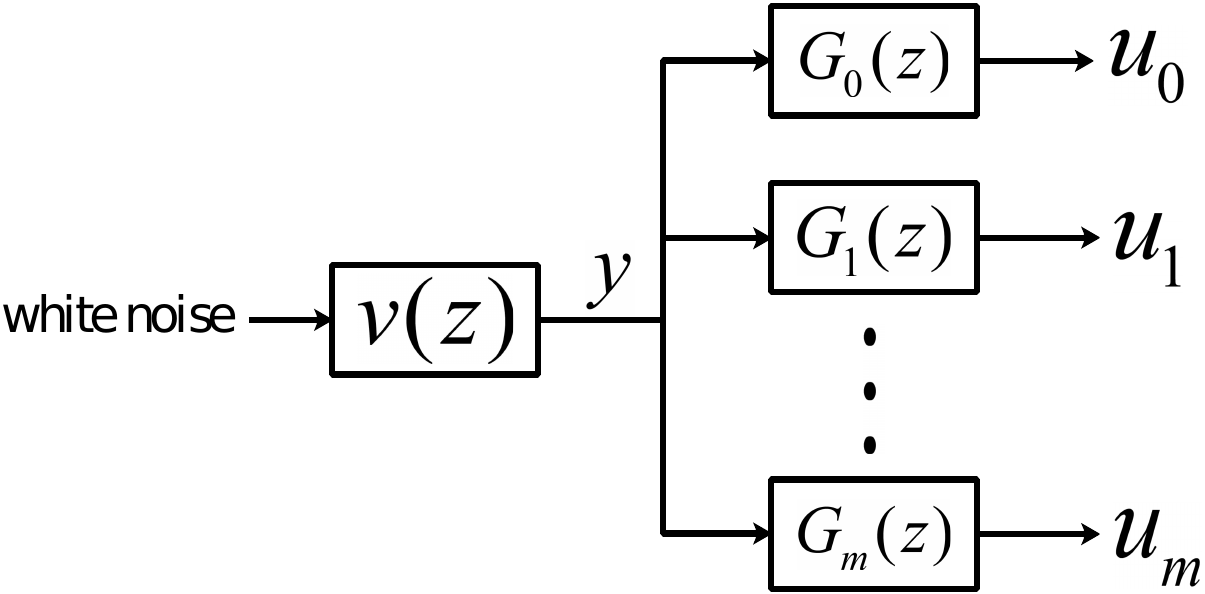}
	\caption{The bank of filters}
	\label{fig:box1}
\end{figure}

With $u$ the output vector of the bank of filters, an estimate of the covariance matrix $\Sigma:=\mathbb{E}\{u(t)u^{*}(t)\}$ yields the matrix $W$ in \eqref{WWj}
 by solving the Lyapunov equation
\eqref{Pick} in Proposition~\ref{LyapunovW}. Then using the homotopy continuation algorithm in subsection~\ref{scalarhomotopy}, we obtain a solution to estimation problem. In fact, using the  $\sigma(z)$ with zeros $0.9e^{\pm2.6i}$,  $0.94e^{\pm1.6i}$, $0.5e^{\pm1.3i}$ and $0.3$, we see in Fig.~\ref{trajectory} how the trajectories of the poles, i.e., the zeros of $a(p(\lambda))$, move as $\lambda$ varies from $0$ to $1$. The poles for $\lambda =0$ are marked with circles and the poles for $\lambda=1$ by $\times$. The continuity of the trajectory shows the feasibility of the homotopy continuation method.
\begin{figure}[!htp]
	\centering
	\includegraphics[width =0.30\textwidth]{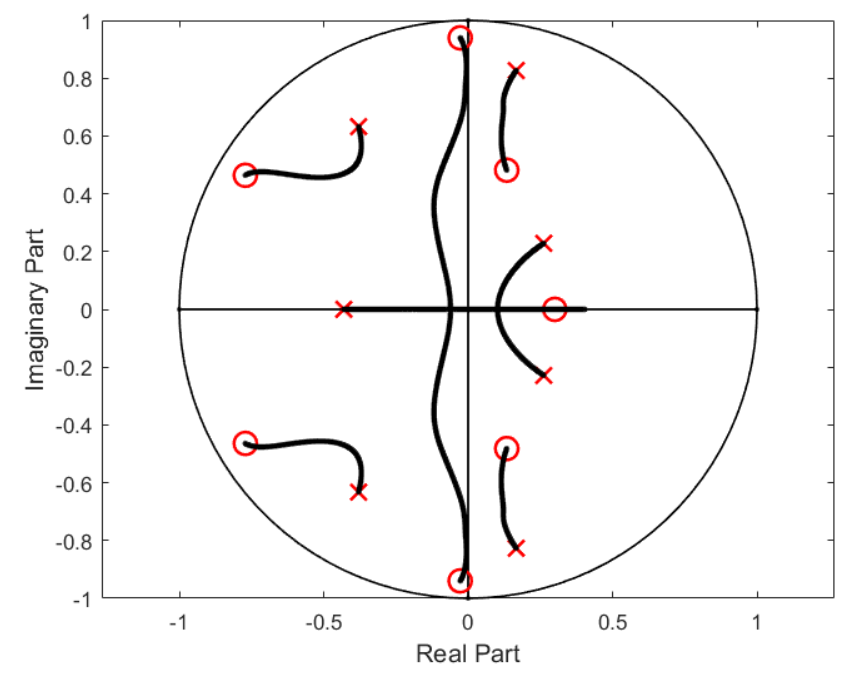}
	\caption{The trajectories of the poles}
	\label{trajectory}
\end{figure}

Moreover, we obtain a solution $P$ of CEE with the singular values
\begin{displaymath}
0.7435,\, 0.1328,\, 0.0794,\, 0.0630,\, 0.0023,\, 0.0003,\, 6\times 10^{-6}, 
\end{displaymath} 
the last three of which are close to zero. Consequently, $P$ has approximately rank 4. Therefore, in view of \eqref{rankP=degf} and the fact that $\deg v=\deg f$, we can reduce the degree of $v(z)$ to 4 to obtain the reduced system $\hat{v}(z)$.  Fig.~\ref{estimate} shows the given spectral factor $v(z)$ together with the degree 7 solution and the approximate degree 4 approximation $\hat{v}(z)$. 
\begin{figure}[!htp]
	\centering
	\includegraphics[width=\columnwidth,height=0.65\columnwidth]{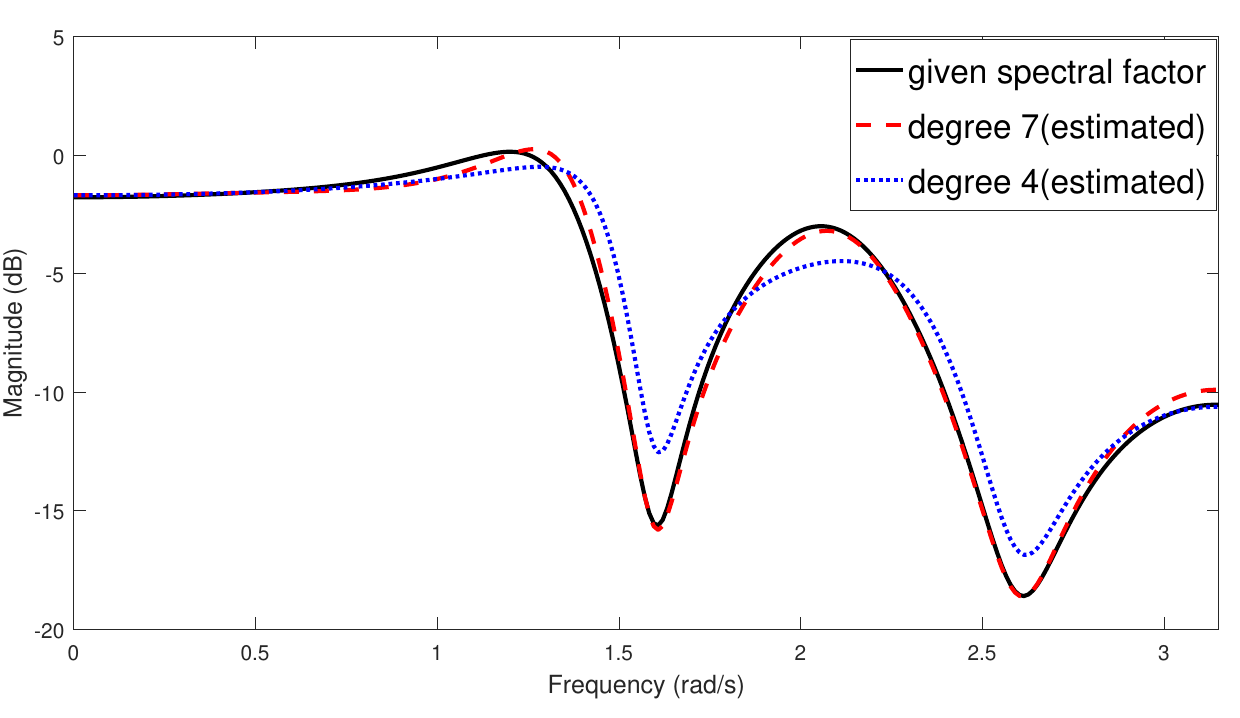}
	\caption{The given spectral factor and its estimated ones}
	\label{estimate}
\end{figure}

More precisely, $\hat{v}(z)=\hat{\rho}\hat\sigma'z)/\hat{a}(z)$
with $\hat\rho=0.5247$ and
$$\hat\sigma(z)=z^{4}+1.5973z^{3}+1.7783z^{2}+ 1.4073 z+0.7157,$$
$$\hat a(z)=z^{4}+0.9341z^{3}+1.112z^{2}+0.7007 z+0.3939,$$ 
where the last three spectral zeros of $\sigma(z)$ have been removed to obtain $\hat\sigma(z)$.
Likewise, computing the degree 5 and  6 approximations show that the corresponding solutions $P$ also have rank approximately 4. 

\subsection{Robust control with sensitivity shaping}

Given a plant 
\begin{equation}
P(z)=\frac{\left(z-1.1 e^{\frac{19}{20} \pi i}\right)\left(z-1.1 e^{-\frac{19}{20} \pi i}\right)}{z(z-1.1)\left(z^{2}+1.21\right)}
\end{equation}
and the feedback configuration in Fig.~\ref{robust},
\begin{figure}[!htp]
	\centering
	\includegraphics[width=0.32\textwidth]{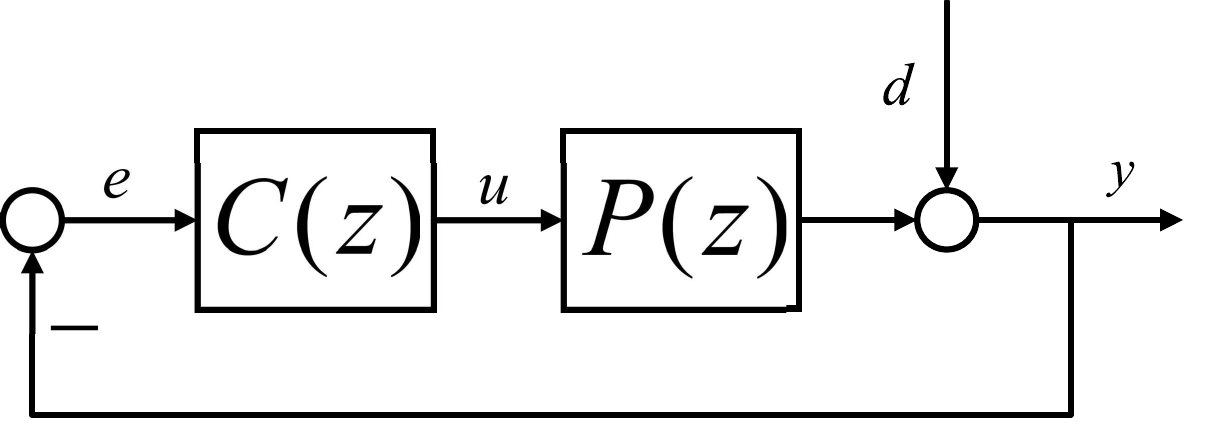}
	\caption{A feedback configuration}
	\label{robust}
\end{figure}
we need to find a controller $C$ such that the system is internally stable and satisfies the following specifications:
\begin{equation}\label{specification}
\begin{split}
&\left|S\left(e^{i \theta}\right)\right| \leq-1 \mathrm{dB},  \quad\theta \in[0,0.3](\mathrm{rad} / \mathrm{sec}) \\
&\left|S\left(e^{i \theta}\right)\right| \leq 0.5 \mathrm{dB}, \quad \theta \in[2.5, \pi](\mathrm{rad} / \mathrm{sec})  
\end{split}
\end{equation}
$$
\|S\|_{\infty}<5 \approx 13.98 \mathrm{dB}
$$	 
where
$$
S(z) :=1 /(1+P(z) C(z))
$$
is the sensitivity function. From the robust control literature \cite{DFT} we know that a necessary and sufficient condition for internal stability is that we have no unstable pole-zero cancellation between $P$ and $C$ in the sensitivity function and that the
sensitivity function is stable.

The plant $P(z)$ has three real unstable poles at $\pm 1.1i$ and $1.1$ and three unstable zeros at $\infty$ and $1.1 e^{\pm \frac{19}{20} \pi i}$ with multiplicities two, one, and one respectively. Since the system should be internally stable, the sensitivity function must satisfy the interpolation conditions
\begin{equation}\label{constraints}
\begin{split}
&S(\pm 1.1 i)=0,\quad S(1.1)=0\\
&S(\infty)=1,\quad S^{\prime}(\infty)=0,\quad S\left(1.1 e^{\pm \frac{19}{20} \pi i}\right)=1
\end{split}
\end{equation}
Since $\|S\|_{\infty}<5$, the function $g(z):=S(z)/5$ maps the exterior of the disc into the unit disc, so 
$$f(z):=\frac{1+g(z^{-1})}{1-g(z^{-1})}=\frac{5+S(z^{-1})}{5-S(z^{-1})}$$
maps the disc into the right half plane, and hence $f$ is a Carath{\'e}odory function.  To find such a function $f$ satisfying the given specifications \eqref{specification} and interpolation constraints \eqref{constraints} is an analytic interpolation problem of the type stated in Section~\ref{sec:intro}.  Since there are seven interpolation conditions, we can construct an interpolant of degree six by choosing six spectral zeros.

Note that the zeros of $f(z)+f(z^{-1})$ are the zeros of 
$$\Gamma(z):=25-S(z)S(z^{-1}).$$
Next we will show how to achieve the given specifications by choosing suitable zeros of $\Gamma(z)$.
Suppose $\Gamma(z)$ has one spectral zero $\lambda$ near $z=e^{i\theta}$, then $\left|S\left(e^{i \theta}\right)\right| \approx 5$ by the continuity of $\Gamma(z)$ at $z=e^{i \theta}$. So by choosing a spectral zero near $z=e^{i \theta}$, we can elevate the frequency response of $S$ at $\theta$ to about 5. More details can be found in \cite{Nagamune}.

 When we choose the spectral zeros at $0.98 e^{ \pm \frac{7}{15} \pi i}$, $0.97 e^{ \pm \frac{1}{2} \pi i}$, $0$ and $-0.1$. we obtain the sensitivity function and the controller as
 \begin{equation*}
S(z)=\frac{\splitfrac{z^6 -0.0414z^5+1.1873z^4-0.8951z^3}{-0.4795z^2-1.0224z-0.5470}}{\splitfrac{z^6 -0.0414z^5+1.5522z^4-0.0209z^3}{+0.5729z^2+0.0192z-0.0219}},
\end{equation*}
$$
C(z)=\frac{0.3648 z^3 + 0.08142 z^2 + 0.434 z}{z^3 + 1.059 z^2 + 1.142 z + 0.411},
$$
respectively.
The frequency response of $S$ is illustrated in Fig.~\ref{figure1}, from which we can see that the specifications are indeed fulfilled.
\begin{figure}[!htp]
	\centering
	\includegraphics[width=0.40\textwidth]{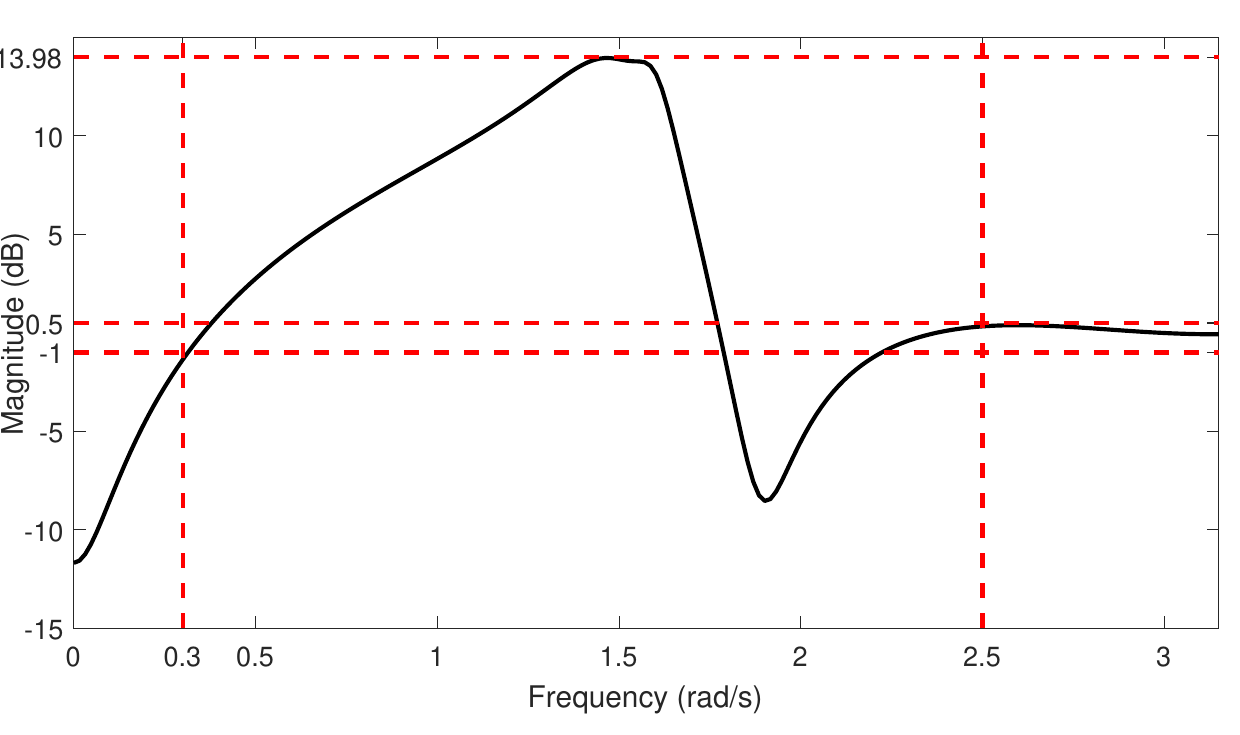}
	\caption{The frequency response of $ S $ which satisfies spec.}
	\label{figure1}
\end{figure}

\subsection{Model reduction in multivariable case}

Consider a system with a $2\times 2$ transfer function 
\begin{equation}\label{truesystem}
V(z)=A(z)^{-1}\Sigma(z)R
\end{equation}
of dimension ten and with observability indices $t_1=t_2=5$,
where
\begin{displaymath}
R=\begin{bmatrix}
2&1\\
1&2
\end{bmatrix},
\qquad
A=\begin{bmatrix}A_{11}&A_{12}\\A_{21}&A_{22}\end{bmatrix}
\end{displaymath}
with
\begin{align*}
A_{11} & = z^{5} - 0.11 z^{4} - 0.08 z^{3} + 0.05 z^{2} - 0.05 {z} - 0.13 \\
A_{12} & = -0.02 z^{4} - 0.15 z^{3} + 0.1 z^{2} - 0.09 {z} - 0.09\ \\
A_{21} & = 0.11 z^{4} + 0.09 z^{3} - 0.03 z^{2} - 0.1 z + 0.12 \\
 A_{22} & = z^{5} + 0.07 z^{4} + 0.19 z^{3} - 0.03 z^{2} - 0.13 z + 0.05,
\end{align*}
and
\begin{equation*}
\Sigma(z)=(z-0.1)(z-0.9)(z-0.37)(z+0.4)(z+0.95)I_{2}.
\end{equation*}
Fig.~\ref{multi} shows the location of poles and zeros \begin{figure}[!thp]
	\centering
	\includegraphics[width=0.30\textwidth]{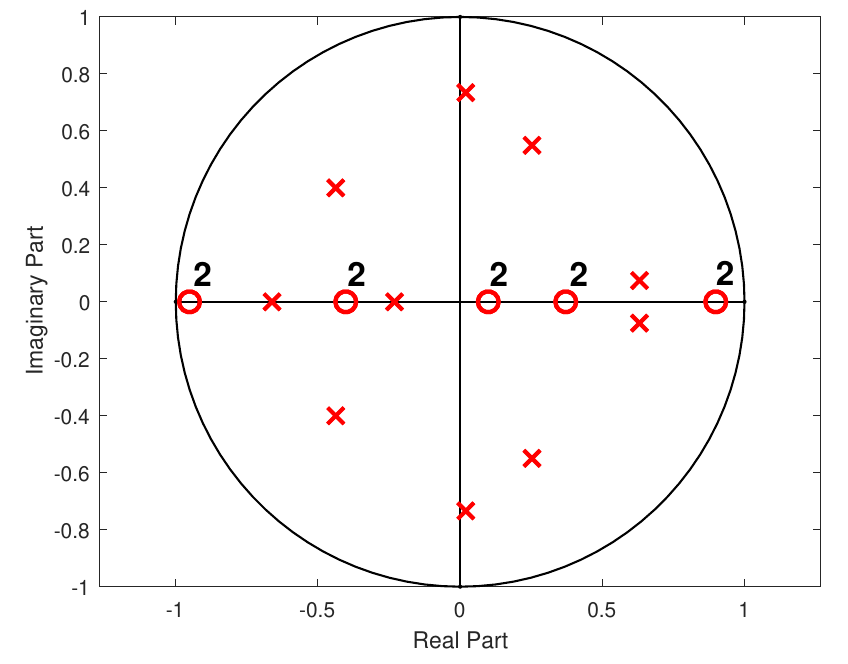}
	\caption{The locations of poles and zeros of $V(z)$}
	\label{multi}
\end{figure}
("2" means there are two zeros at the same position). Clearly there is no pole zero cancellation.
Let $F$ be the matrix-valued Carath\'eodory function $F(z):= \Phi_+(z^{-1})$, where $\Phi_+$ is the positive real function satisfying \eqref{VPhi+}.

Next, passing normalized (vector-valued) white noise through the filter
\[
\text{white noise}\stackrel{u}\negmedspace{\longrightarrow}\fbox{$V(z)$}\negmedspace\stackrel{y}
{\longrightarrow} 
\]
with transfer function $V(z)$, we generate a vector-valued stationary process $y$ with an observed record $y_0, y_1,y_2, \dots,y_N$, and from this output data we estimate the $2\times 2$ matrix valued covariance sequence 
\begin{equation}
\hat{C}_{k}=\frac{1}{N-k+1}\sum_{t=k}^{N}y_{t}y_{t-k}' .
\end{equation}
We want to determine a matrix-valued Carath\'eodory function $F$ satisfying the interpolation conditions
\begin{equation}
\label{ }
 \frac{1}{k!}F^{(k)}(0)=\hat{C}_{k},\quad k=0,1,\cdots,5  .
 \end{equation}
This is a matrix-valued covariance extension problem, which takes the form \eqref{multinterpolation} with  $\ell=2$, $m=0$ and $n_0=6$. Using the homotopy method of subsection~\ref{sec:multialgorithm}, the poles move as $\lambda$ varies from $0$ to $1$ as shown in Fig.~\ref{trajemulti}.
\begin{figure}[!htp]
	\centering
	\includegraphics[width=0.32\textwidth]{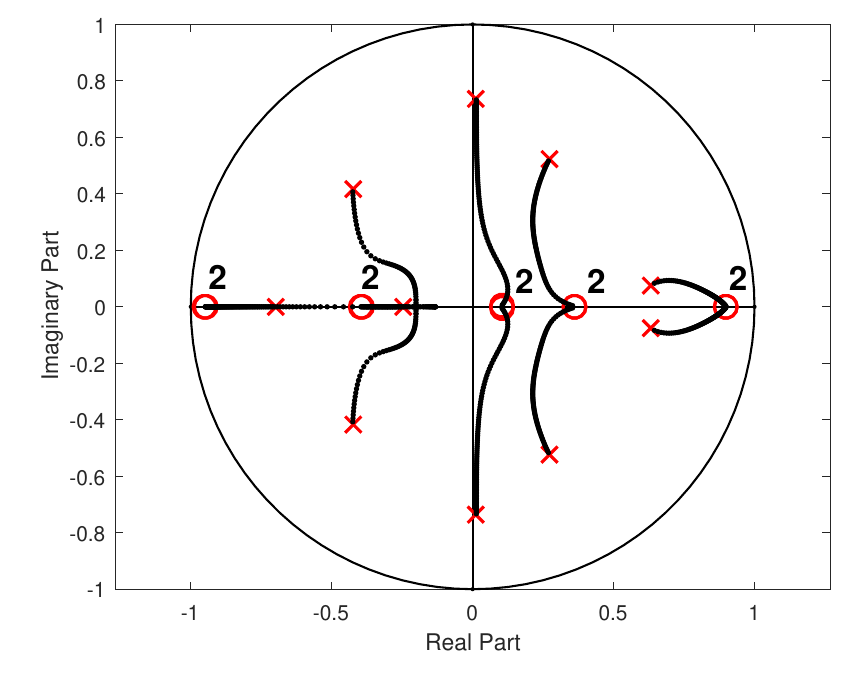}
	\caption{The trajectory of poles as $\lambda$ varies from $0$ to $1$}
	\label{trajemulti}
\end{figure}

The modified Riccati equation has a solution $P$ with eigenvalues
\begin{equation*}
\begin{split}
&5.8\times10^{-6},2.03\times10^{-4},1.8\times10^{-3},3.9\times10^{-3},6\times10^{-3},\\
&0.03, 0.365,0.4879,0.7895,0.8967
\end{split}
\end{equation*}
The first six eigenvalues are very small, so we can reduce the degree of this system from 10 to  4 by choosing the first three covariance lags $\hat{C}_0,\hat{C}_1,\hat{C}_2$  and removing six zeros of $\Sigma(z)$. We choose two double zeros at $0.9$ and $-0.95$. The reduced system 
$$\hat{V}(z)=H\hat{A}(z)^{-1}\hat\Sigma(z)\hat R$$
has observability indices $t_1=t_2=2$, and
\begin{equation*}
H=\begin{bmatrix}
535/378  &     -363/3758\\
-363/3758  &     891/523
\end{bmatrix},\end{equation*}
\begin{displaymath}
\hat{R}=\begin{bmatrix}
    1623/1138  &    1177/1625  \\
1997/3448    &  7279/6408 
\end{bmatrix},
\end{displaymath}
\begin{equation*}
\hat{A}(z)=\begin{bmatrix}
z^2 - 0.01968 z + 0.09216&-0.1574 z - 0.08796\\
0.03346 z + 0.1083&z^2 + 0.1314 z + 0.4366
\end{bmatrix},
\end{equation*}
\begin{equation*}
\hat{\Sigma}(z)=(z-0.9)(z+0.95)I_{2}.
\end{equation*}
The singular values of the true system \eqref{truesystem} together with those of the estimated systems of degree 10 and 4 are shown in Fig.~\ref{figure2}. 
\begin{figure}[!htp]
	\centering
	\includegraphics[width = 0.38\textwidth]{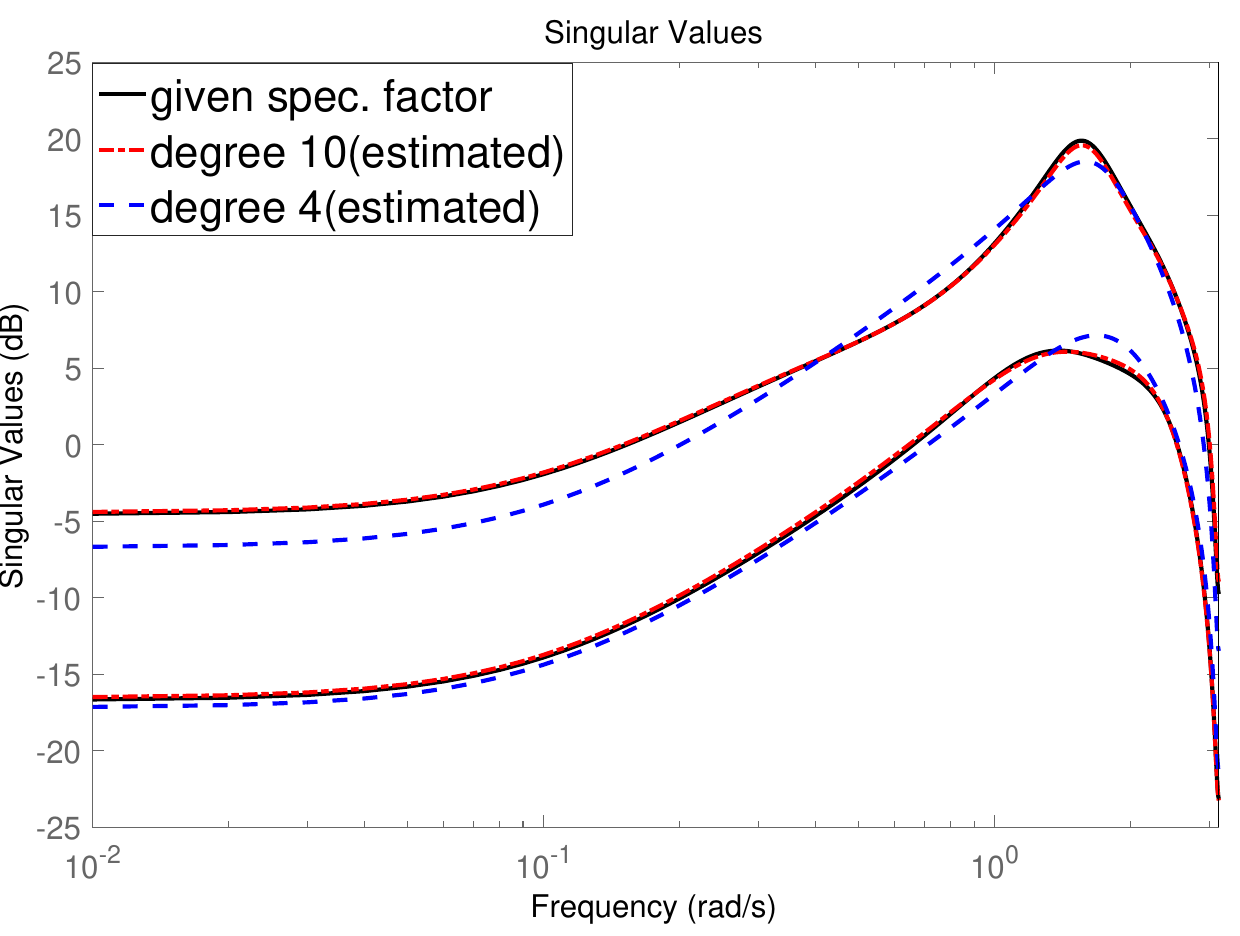}
	\caption{Estimated singular values and the true ones}
	\label{figure2}
\end{figure}

 Here the estimated degree 10 system estimates the true system \eqref{truesystem} perfectly, as the black curves of the given spectral factor  are completely covered by the red estimate curves. However, the estimated system of degree 4 approximates the true system well. 
\begin{figure*}[!htp]
	\centering
	\includegraphics[width = 0.70\textwidth]{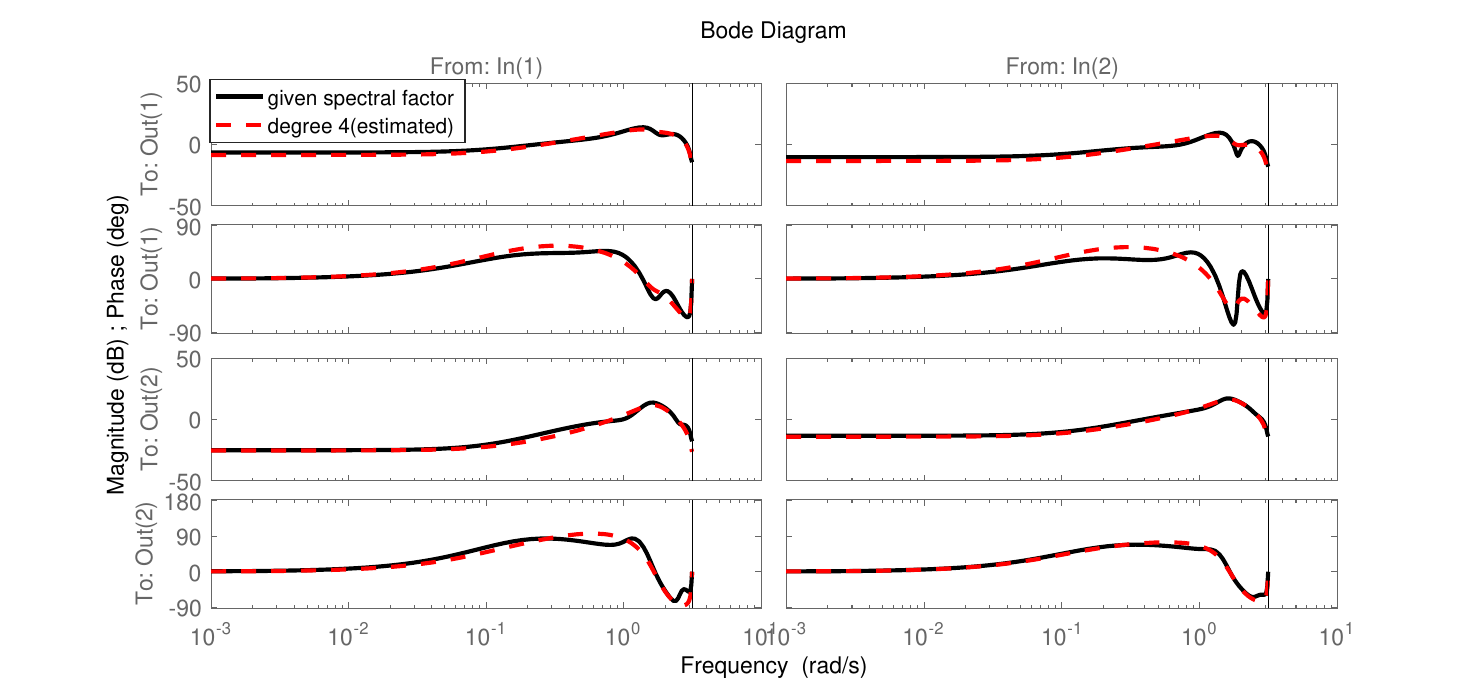}
	\caption{Bode plot}
	\label{figure4}
\end{figure*}
Fig.~\ref{figure4} plots the magnitude and the phase of the frequency response of each input/output pair in the true spectral factor and estimated one with degree 4.  The first column of plots shows the response from the first input to each output. The second column shows the response from the second input to each output. The first and the third line are the magnitudes of the frequency response, and the second and the fourth line are the phases of the frequency response.

\section{concluding remarks}\label{sec:conclusions}

We have shown that the modified Riccati equation  introduced in \cite{BLpartial} for solving the covariance extension problem can be used for very general analytic interpolation problems (with both rationality and derivative constraints) by merely changing certain parameters computed from data. A robust and efficient numerical algorithm based on homotopy continuation has been provided. There are still some open questions in the multivariable case. The most general formulation of the multivariable analytic interpolation with rationality constraints has been marred by difficulties to establish existence and, in particular, uniqueness in the various parameterizations \cite{Gthesis,BLN,G2006,G2007,FPZbyrneslindquist,RFP,Takyar,ZhuBaggio,Zhu}, and we have encountered similar difficulties here. Our approach attacks these problems from a different angle and might put new light on these challenges. Therefore future research efforts will be directed towards settling these intriguing open questions in the context of the modified Riccati equation \eqref{multCEE}.

\appendix

\subsection{Proof of Proposition~\ref{w2uUprop}}

From \eqref{uU}, \eqref{Vmatrix} and \eqref{D} we have 
\begin{equation}
\label{D2u}
u=\begin{bmatrix}0&I_{n}\end{bmatrix}V^{-1}Te,
\end{equation}
where 
\begin{subequations}\label{DDj}
\begin{equation}
\label{D2}
T=\text{diag}\,(D_0,\dots,D_m) =(W+\tfrac{1}{2}I)^{-1}(W-\tfrac{1}{2}I).
\end{equation}
with
\begin{equation}\label{Dj}
D_{j}=\begin{bmatrix}
d_{j0}&~&~&~\\
d_{j1}&d_{j0}&~&~\\
\vdots&\ddots&\ddots&~\\
d_{jn_{j}-1}&\cdots&d_{j1}&d_{j0}
\end{bmatrix}
\end{equation}
\end{subequations}
Consequently,
\begin{equation}
\label{u}
u=Md,
\end{equation}
where $d$ is the $n$-vector 
\begin{displaymath}
d=\begin{bmatrix}d_0'&d_{10}'&d_1'&\cdots&d_{m0}'&d_m'\end{bmatrix}',
\end{displaymath}
 and $M$ is the nonsingular $n\times n$ matrix obtained by deleting the first row and the first column in $V^{-1}$. We want to establish a diffeomorphism $d=\varphi(w)$ from  the $n$-vector \eqref{wdefn}, i.e.,
 \begin{displaymath}
w=(w_0',w_{10}, w_1',w_{20}, w_2',\dots,w_{m0}, w_m')' ,
\end{displaymath}
to $d$. To this end, we compute $D_j$ to obtain
\begin{align*}
D_j:=&(W_{j}+\frac{1}{2}I)^{-1}(W_{j}-\frac{1}{2}I)\\
=&\begin{bmatrix}
 (w_{j0}+\frac{1}{2})^{-1}(w_{j0}-\frac{1}{2})&0\\
 C_j^{-1}w_j(w_{j0}+\frac{1}{2})^{-1}&C_j^{-1}(C_j-I)
  \end{bmatrix}.
 \end{align*}
 Therefore,
 \begin{displaymath}
 \begin{bmatrix}d_{j0}\\d_j\end{bmatrix}= \begin{bmatrix}(w_{j0}+\frac{1}{2})^{-1}(w_{j0}-\frac{1}{2})\\ C_j^{-1}w_j(w_{j0}+\frac{1}{2})^{-1}\end{bmatrix}
 \end{displaymath}
 from which we have $w_j=C_j(w_{j0}+\tfrac12)d_j$ , $w_{j0}=\tfrac12(1+d_{j0})(1+d_{j0})^{-1}$ and
\begin{equation} \label{S}
S_j:=C_j^{-1}(C_j-I)=\begin{bmatrix}d_{j0}&~&~\\\vdots&\ddots&~\\d_{jn_{j}-2}&\cdots&d_{j0}\end{bmatrix}.        
\end{equation}
 Hence, we have the smooth maps
\begin{equation}
\label{wd}
\begin{split}
w_j&=(I-S_j)^{-1}(1-d_{j0})^{-1}d_j\\
d_j&=(w_{j0}+\tfrac12)^{-1}C_j^{-1}w_j
\end{split}
\end{equation}
defining a diffeomorphism $d=\varphi(w)$ from $w$ to $d$. Thus, since the matrix $M$ in  \eqref{u} is nonsingular,  $u=M\varphi(w)$ is the sought diffemorphism $\omega$.

Finally, it follows from \eqref{DDj} that there is a linear map $N$ such that $D=N(d)=N(M^{-1}u)$, and hence there is a linear map $L$ such that  $U=Lu$, as claimed.

\bibliography{ifacconf}    

\begin{thebibliography}{0}

\bibitem{Kalman-81}
Kalman, R.~E., Realization of covariance sequences, {\em Proc. Toeplitz Memorial Conference}, Tel Aviv, Israel, 1981.

\bibitem{Gthesis}
Georgiou, T.~T., {\em Partial realization of covariance sequences},  PhD thesis, {CMST}, Univ. Florida, 1983.

\bibitem{G87}
Georgiou, T.~T., Realization of power spectra from partial covariances, 
{\em IEEE Trans. Acoustics, Speech and Signal Processing},  35:438--449, 1987.

\bibitem{BLGuM}
Byrnes, C. I. and  Lindquist, A., Gusev, S. V. and Matveev, A. S.,
A complete parameterization of all positive rational extensions of a covariance sequence,
 {\em IEEE Trans. Automatic Control}, 1995, 40:1841--1857.
 
 \bibitem{BLpartial}
Byrnes C I and Lindquist, A., On the partial stochastic realization problem, {\em IEEE Trans. Automatic Control}, 1997, 42(8): 1049-1070.

\bibitem{BGuL}
Byrnes, C.~I., Gusev, S.~V.  and Lindquist, A., 
A convex optimization approach to the rational covariance extension problem.
{\em SIAM J. Control and Optimization}, 37:211--229, 1999.


\bibitem{b12} Delsarte P, Genin Y and Kamp Y,  Speech modeling and the trigonometric moment problem, {\em Philips J. Res}, 1982, 37(5/6): 277-292.

\bibitem{b13} Lindquist A and Picci G. Canonical correlation analysis, approximate covariance extension, and identification of stationary time series, {\em Automatica}, 1996, 32(5): 709-733.

\bibitem{LPbook}
Lindquist A and Picci G., {\em Linear stochastic systems: A Geometric Approach to Modeling, Estimation and Identification}, Springer, 2015.

\bibitem{b15} Georgiou T T, A topological approach to Nevanlinna-Pick interpolation, {\em SIAM J. Math. Anal.}, 1987, 18(5): 1248-1260.

\bibitem{b1} Byrnes CI, Georgiou TT and Lindquist A,  A generalized entropy criterion for Nevanlinna-Pick interpolation with degree constraint, {\em IEEE Trans. Autom. Contr.}, 2001, 46: 822-839.

\bibitem{BLkimura}
Byrnes C I, and Lindquist, A., A convex optimization approach to generalized moment problems, in  {\em Control and Modeling of Complex Systems: Cybernetics in the 21st Century}, Koichi H. et. al (eds.), Birkh{\"a}user, 2003. 

\bibitem{DFT}
Doyle, J.C., Francis, B.A. and Tannenbaum, A.R., {\em Feedback Control Theory}, MacMillan, 1992.

\bibitem{b2} Byrnes C L, Georgiou T T and Lindquist A, A new approach to spectral estimation: A tunable high-resolution spectral estimator, {\em IEEE Trans. Signal Proc.}, 2000, 48: 3189-3205.

\bibitem{GL1}
T.~T. Georgiou and A.~Lindquist, ``Kullback-Leibler approximation of spectral
  density functions,'' \emph{IEEE Transactions on Information Theory}, vol.~49,
  no.~11, pp. 2910--2917, 2003.

\bibitem{Ghosh} B.J. Ghosh, An approach to simultaneous system design, Part II: Nonswitching gain and dynamic feedback compensation by algebraic geometric methods, {\em SIAM J. Control and Optimization}

\bibitem{b9} Delsarte P, Genin Y and Kamp Y, On the role of the Nevanlinna-Pick problem in circuit and system theory,  {\em Intern. J. Circuit Theory and Applications}, 1981, 9(2): 177-187.

\bibitem{b10} Youla D C and Saito M, Interpolation with positive real functions,  {\em J. Franklin Institute}, 1967, 284(2): 77-108.

\bibitem{blomqvist}
Blomqvist A and Nagamune R, An extension of a Nevan\-linna-Pick interpolation solver to cases including derivative constraints, {\em IEEE Conf. Dec. Control},  2002, 3: 2552-2557.

\bibitem{GreenLimebeer}
Green, M. and Limebeer, D.J.N., {\em Linear Robust Control}, Prentice Hall, 1995. 

\bibitem{b6} Lindquist A, Partial Realization Theory and System Identification Redux, {\em Proc. 11th Asian Control Conference}, Gold Coast, Australia, Dec. 17-20, 2017, pp. 1946-1950.

\bibitem{CLccdc} Cui, Y. and Lindquist, A, A modified Riccati approach to analytic interpolation with applications to system identification and robust control, {\em Proc. Chinese Conference on Decision and Control}, Nanchang, June 3-5, 2019.

\bibitem{CLcdc19} Cui, Y. and Lindquist, A, Multivariable analytic interpolation with complexity constraints: A modified Riccati approach, {\em Proc. Proc. 58th IEEE Conference on Decision and Control} (CDC2019), Nice, France.

\bibitem{BFL} Byrnes C I, Fanizza G and Lindquist A,  A homotopy continuation solution of the covariance extension equation,  in {\em New Directions and Applications in Control Theory}, Springer, Berlin, Heidelberg, 2005: 27-42.

\bibitem{Takyar}
Takyar. M S and Georgiou, T T, Analytic interpolation with a degree constraint for matrix-valued functions,  {\em IEEE Trans. Automatic Control}, 2010, 55(5): 1075-1088.

\bibitem{SIGEST} C. I. Byrnes, S. V. Gusev and A. Lindquist, From finite covariance windows to modeling filters: A convex
optimization approach, SIAM Review, vol. 43, N0. 4, Dec. 2001, 645--675. 

\bibitem{BEL} C. I. Byrnes,  Enqvist, P. and A. Lindquist, Identifiability and well-posedness of shaping-filter parameterizations: A global analysis approach, SIAM Journal on Control and Optimization, vol. 41, no. 1, 2002, pp. 23-59.

\bibitem{Kalmanbook} R. E. Kalman, P. L. Falb and M. A. Arbib, Topics in Mathematical Systems Theory, McGraw-Hill, 1969.

\bibitem{Aoki} M. Aoki, State Space Modeling of Time Series, Springer-Verlag, 1987.

\bibitem{OverscheeDeMoor93}
P. Van Overschee and B. De Moor, Subspace algorithms for stochastic identification problem, Automatica, vol. 3, 1993, pp. 649-660.

\bibitem{OverscheeDeMoor96} P. Van Overschee and B. De Moor, Subspace Identification for Linear Systems: Theory Implementation Applications, Kluwer Academic Publishers, 1996.

\bibitem{LP96}  A. Lindquist and G. Picci, Canonical correlation analysis, approximate covariance extension, and identification of stationary time series, Automatica, vol. 32, no. 5, 1996, pp. 709-733.

\bibitem{DLM} A. Dahl\'{e}n, A. Lindquist and J. Mari, Experimental evidence showing that stochastic subspace identification methods may fail, Systems and Control Letters, vol. 34,  1998, pp. 303-312.

\bibitem{GraggL} W. B. Gragg and A. Lindquist, On the partial realization problem, Linear Algebra and Applications, vol. 50, 1983, pp. 277-319.

\bibitem{Kalmanprivate} R. E. Kalman, private communication, 1972. 

\bibitem{b16} Georgiou T T, Spectral estimation via selective harmonic amplification, {\em  IEEE Trans. Aut. Contr.}, 2001, 46(1): 29-42.

\bibitem{Higham}
Higham N J, {\em Functions of Matrices: Theory and Computation}, SIAM, 2008.

\bibitem{b8} Byrnes C I and Lindquist A. On the duality between Filtering and Nevanlinna-Pick interpolation, {\em SIAM Journal on Control and Optimization}, 2000, 39(3): 757-775.

\bibitem{KLR}
Karlsson, J., Lindquist, A. and Ringh, A., The multidimensional moment
  problem with complexity constraint, {\em Integral Equations and Operator Theory},
  84 (2016), pp.~395--418.
  
  \bibitem{Nagamune}
  Nagamune R, Sensitivity Reduction for SISO Systems Using the Nevanlinna-Pick Interpolation with Degree Constraint, {\em Proc. 14th International Symposium of the Mathematical Theory of Networks and Systems (MTNS)}, 2000. 
  
  \bibitem{AllgowerGeorg}
Allg{\"o}wer,  E.L. and Georg, K, {\em Numerical Continuation Method, An Introduction}, Springer-Verlag, 1990.

\bibitem{BLN}
Blomqvist A, Lindquist A, and Nagamune R, Matrix-valued Nevanlinna-Pick interpolation with complexity constraint: An optimization approach, {\em IEEE Trans. Automatic Control}, 2003, 48(12): 2172-2190.

\bibitem{G2006}
Georgiou, T~T, Relative entropy and the multivariable moment problem, {\em IEEE Transactions on Automatic Control}, 2006, 52(3):1052-1066.

\bibitem{G2007}
Georgiou, T~T, The Carath{\'e}odory-Fej{\'e}r-Pisarenko decomposition and its multivariable counterpart, {\em IEEE Transactions on Automatic Control}, 2007, 52(2):212-228.

\bibitem{FPZbyrneslindquist}
Ferrante, A, Pavon, M, and  Zorzi, M, Application of a global inverse function theorem of Byrnes and Lindquist to a multivariable moment problem with complexity constraint, in {\em Three Decades of Progress in Control Sciences}, X. Hu et al. (Eds.), 2010, Springer, pp. 153-167.

\bibitem{RFP}
Ramponi, F, Ferrante, A and Pavon, M,  A globally convergent matrical algorithm for multivariate spectral estimation,
{\em IEEE Trans. Automatic Control}, 2009, 54(10): 2376--2388.

\bibitem{ZhuBaggio}
Zhu; B and  Baggio, G, On the existence of a solution to a spectral estimation problem {\'a} la Byrnes-Georgiou-Lindquist, {\em IEEE Transactions on Automatic Control}, 2019.

\bibitem{Zhu}
Zhu, B, On a parametric spectral estimation problem, arXiv preprint arXiv:1712.07970, 2018.





\end{thebibliography}

\begin{IEEEbiography}[{\includegraphics[width=1in,height=1.25in,clip,keepaspectratio]{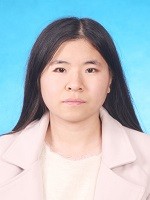}}]
{\bf Yufang Cui} (S'19) received the B.S. degree in automation from Northwestern Polytechnical University, Xi'an, China, in 2017, and the M.S. degree in control engineering from Shanghai Jiao Tong University, Shanghai, China, in 2020.
	Her research involves analytic interpolation theory with complexity constraints and its applications in control and system identification.
	
	Yufang Cui was a Finalist for the Zhang Si-Ying Outstanding Youth Paper Award at the 31st Chinese Control and Decision Conference (CCDC2019).
\end{IEEEbiography}

\begin{IEEEbiography}[{\includegraphics[width=1in,height=1.25in,clip,keepaspectratio]{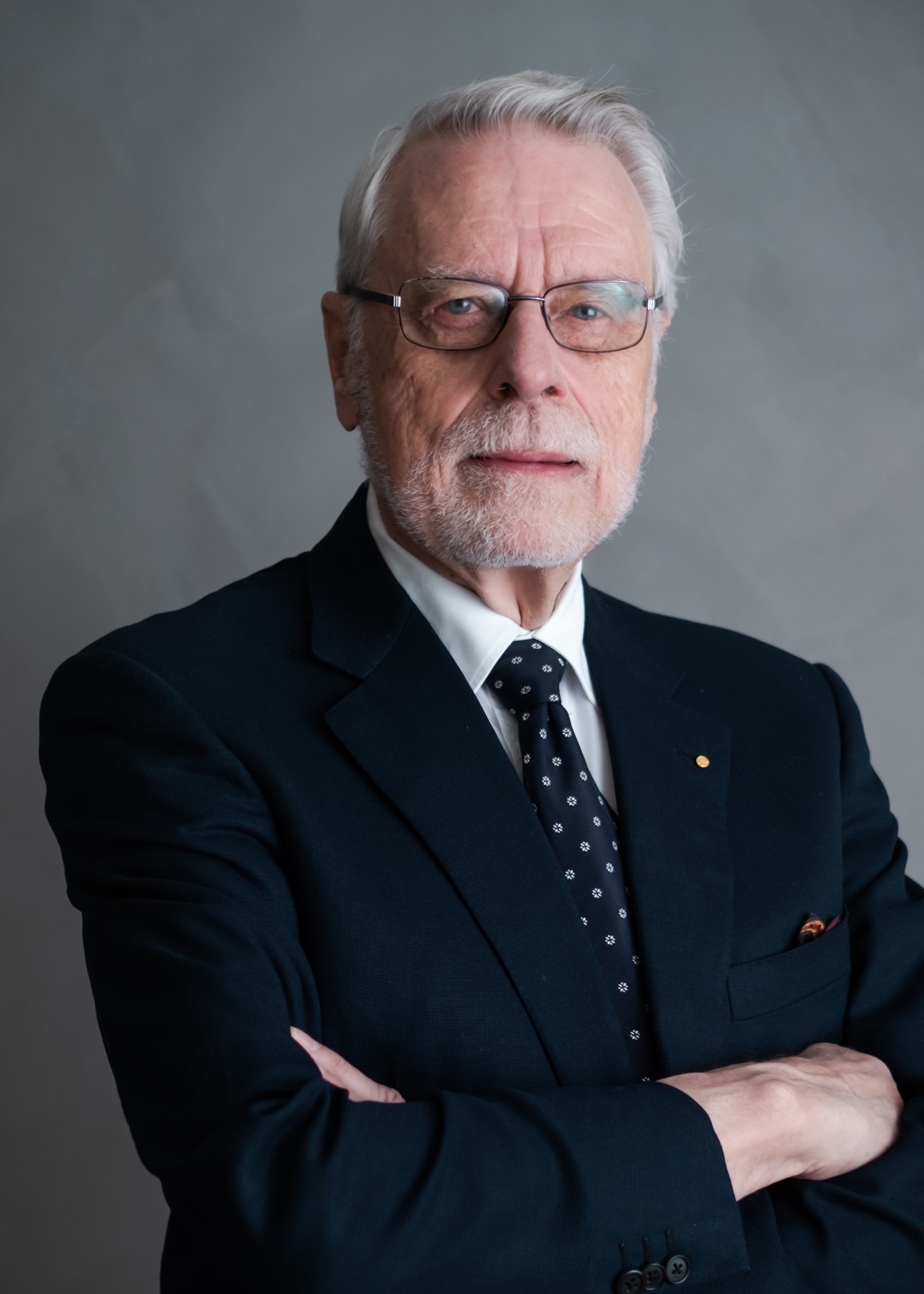}}]{Anders Lindquist}
(M’77--SM’86--F’89--LF’10) received the Ph.D. degree in optimization and systems theory from the Royal Institute of Technology, Stockholm, Sweden, in 1972, and an honorary doctorate (Doctor Scientiarum Honoris Causa) from Technion (Israel Institute of Technology) in 2010.

He is currently a Zhiyuan Chair Professor at Shanghai Jiao Tong University,  China, and Professor Emeritus at the Royal Institute of Technology (KTH), Stockholm, Sweden. Before that he had a full academic career in the United States, after which he was appointed to the Chair of Optimization and Systems at KTH.

Dr. Lindquist is a Member of the Royal Swedish Academy of Engineering Sciences, a Foreign Member of the Chinese Academy of Sciences, a Foreign Member of the Russian Academy of Natural Sciences, a Member of Academia Europaea (Academy of Europe), an Honorary Member the Hungarian Operations Research Society, a Fellow of SIAM, and a Fellow of IFAC. He received the 2003 George S. Axelby Outstanding Paper Award, the 2009 Reid Prize in Mathematics from SIAM, and the 2020 IEEE Control Systems Award, the IEEE field award in Systems and Control.
\end{IEEEbiography}

\end{document}